\documentclass[12pt, twoside, leqno]{article}
\usepackage{hyperref}

\usepackage{amsmath,amsthm}
\usepackage{amssymb}

\usepackage{enumitem}

\usepackage{graphicx}

\usepackage[T1]{fontenc}

\usepackage{amsfonts}
\usepackage{float}
\usepackage{cancel}
\usepackage{array}

\newtheorem{theorem}{Theorem}[section]

\newtheorem{lemma}[theorem]{Lemma}

\theoremstyle{definition}

\numberwithin{equation}{section}

\begin{document}


\baselineskip=17pt


\title{The extremal values of the ratio of differences of power mean, arithmetic mean, and geometric mean}

\author{Yagub N. Aliyev\\
School of IT and Engineering\\ 
ADA University\\
Ahmadbey Aghaoglu str. 61 \\
Baku 1008, Azerbaijan\\
E-mail: yaliyev@ada.edu.az}
\date{}

\maketitle


\renewcommand{\thefootnote}{}



\renewcommand{\thefootnote}{\arabic{footnote}}
\setcounter{footnote}{0}


\begin{abstract}
In the paper the maximum and the minimum of the ratio of the difference of the arithmetic mean and the geometric mean, and the difference of the power mean and the geometric mean of $n$ variables, are studied. A new optimization argument was used which reduces $n$ variable optimization problem to a single variable. All possible cases of the choice of the power mean and the choice of the number of variables of the means are studied. The obtained results generalize and complete the earlier results which were either for specific intervals of power means or for small number of variables of the means. Some of the results are formulated as the best constant inequalities involving interpolation of the arithmetic mean and the geometric mean. The monotonicity and convergence of these best constants are also studied.
\end{abstract}

\section{Introduction}

Arithmetic mean, Geometric mean and their generalizations such as Power mean are fundamental concepts of Analysis and the inequalities comparing these means have both theoretical and applicational importance \cite{beck}, \cite{bull}, \cite{mitrin}. Arithmetic Mean - Geometric Mean (AM-GM) inequality is one of the many similar inequalities and it was generalized in many different ways (see e.g. \cite{hardy}, Chapter II). There are many different methods to prove these inequalities. One of them is Maclaurin's (geometric) and Cauchy's (analytical) proof of AM-GM inequality. Both methods are based on the idea of optimizing the product of any two numbers while keeping the sum of the same two numbers constant (see \cite{hardy}, p. 19, footnote). In the current paper we will use an analogous method to optimize the ratio of differences of some means but with 3 variables. As a corollary, these extreme values will give us the best constants for some generalizations of AM-GM inequality.

We need some notations to continue. Let $P_\alpha$ be power mean of non-negative real numbers $x_1,\ldots,x_n$:
$$
P_\alpha(x_1,\ldots,x_n)=\left(\frac{\sum\limits_{i =
1}^n{x_{i}^\alpha}}{n}\right)^\frac{1}{\alpha},
$$
except when (i) $\alpha=0$ or (ii) $\alpha<0$ and one or more of the $x_i$ are zero. If $\alpha=0$, then we assume that $P_\alpha=G_n$, where $G_n=\sqrt[n]{\prod\limits_{i =
1}^n{x_{i}}}$ is the geometric mean of numbers $x_1,\ldots,x_n$. If $\alpha<0$ and there is an $i$ such that $x_i=0$, then we assume that $P_\alpha=0$ (see \cite{hardy}, p. 12). In particular, if $\alpha=-1,\ 1,\ 2$, then we obtain familiar harmonic, arithmetic, and quadratic means:

$$H_n=\frac{n}{\sum\limits_{i =
1}^n{\frac{1}{x_{i}}}},\ 
A_n=\frac{\sum\limits_{i =
1}^n{x_{i}}}{n},\ Q_n=\sqrt{\frac{\sum\limits_{i =
1}^n{x_{i}^2}}{n}}.
$$

In the current paper, we are interested in finding extremal values of $\frac{A_n-G_n}{P_\alpha-G_n}$ for all possible choices of $\alpha$. For $n=2$ this was done in \cite{wu} (p. 153, Theorem 1), \cite{kouba} (p. 926, Theorem 3.3, $p=1$, $t=1$, $s\rightarrow0$), \cite{james} (p. 276, Theorem 1, only maximum). For general $n$ some intervals of values of $\alpha$ were studied by Wu Shanhe in \cite{wu0}, and Wen Jiajin, Cheng Sui-Sun, Gao Chaobang in \cite{wen2} and their references. The case $\alpha=-1$ follows from the detailed study of $\frac{A_n-G_n}{A_n-H_n}$ by N. Lord in \cite{lord}. Many cases such as general $\alpha<0$ or e.g. $\frac{1}{2}\le\alpha<1$ and $n\ge\frac{1}{1-\alpha}$ (see Theorem 1 below) remained open for at least one of the extreme values (see \cite{wu0}, p. 646-647 and \cite{wen2}, p. 134 and the references [5] and [6] in \cite{wen2}) and the current paper intends to consider all of them in a unified way. Similar questions for weighted means were studied in \cite{dinh} and its references. For related problems involving harmonic mean, quadratic mean and other exotic means such as the first and the second Seiffert means, the identric mean, the Neuman-S\'andor mean, the Sándor–Yang mean, the logarithmic mean, the contraharmonic mean, the bivariate means, S\'andor means $X$ and $Y$ etc. of two numbers see \cite{chu0}, \cite{chu}, \cite{chu1}, \cite{chu2}, \cite{kouba}, \cite{kouba1}, \cite{zhao}, \cite{qian}, \cite{long}, \cite{tian}, \cite{he}, \cite{neu}, \cite{mitev1}, \cite{mitev2}, \cite{zhu}, \cite{james0}, \cite{bhayo1}.

In \cite{aliyev1} it was shown that $-\infty<\frac{A_n^n-G_n^n}{H_nA_n^{n-1}-G_n^n}\le\frac{\lambda_n}{n^2}$, and the best constant $\lambda_n$ satisfies $\frac{n^3}{n-1}<\lambda_n<\frac{n^3}{n-2}$ for $n>2$. In particular, it follows that since $1+\frac{1}{n-1}<\frac{\lambda_n}{n^2}<1+\frac{1}{n-2}$, the ratio $\frac{\lambda_n}{n^2}$ decreases for $n>2$ and approaches $1$. In the current paper it is shown that similar monotonicity and convergence results hold true for the maximum and minimum of $\frac{A_n-G_n}{P_\alpha-G_n}$.

Applications include Entropy-Transport problems \cite{pon}, Euler, Finsler-Hadwiger and Ali type inequalities for simplices \cite{wu0}, inequalities involving the power of eigenvalues of Hermitian matrices and integral mean involving permanent of complex matrices \cite{wen2}.

\section{Lemmas}

In this section we collect some preparatory results.
We will need the following interesting limit.

\begin{lemma} If $\alpha,\ \beta,\ \gamma,$ and $\delta$ are real numbers and $\gamma\ne\delta$, then
$$
\lim_{(x_1,x_2,\ldots,x_n)\rightarrow (x,x,\ldots,x)} \frac{P_\alpha-P_\beta}{P_\gamma-P_\delta}=\frac{\alpha-\beta}{\gamma-\delta},
$$
where $x>0$ is an arbitrary real number.
\end{lemma}

\begin{proof} Since $\gamma\ne\delta$, by monotonicity of the power mean $P_\gamma\ne P_\delta$, unless all $x_i$ are equal (\cite{hardy}, p. 26). If $\alpha=\beta$, then $P_\alpha=P_\beta$ and the equality is trivially true. So let us assume that $\alpha\ne\beta$. For now  suppose also that $\alpha\beta\gamma\delta\ne0$.

Because of homogeneity, without loss of generality we can assume that $x=1$ and $\sum_{j=1}^nx_j=nx=n$.
We will use multivariable Taylor's formula (see e.g. \cite{konig}, p. 64 ff.; \cite{zorich}, Sect. 8.4.4). We find that if $j,k=1,2,\ldots,n$ and $k\ne j$, then
$$
P_\alpha(1,1,\ldots,1)=1,\ \left(P_\alpha\right)_{x_j}(1,1,\ldots,1)=\frac{1}{n},
$$
$$
\left(P_\alpha\right)_{x_j x_j}(1,1,\ldots,1)=\frac{(\alpha-1)(n-1)}{n^2},\ \left(P_\alpha\right)_{x_j x_k}(1,1,\ldots,1)=-\frac{\alpha-1}{n^2},
$$
Therefore,
$$
P_\alpha=1+\sum_{j=1}^n\frac{1}{n}(x_j-1)+\sum_{j=1}^n\frac{(\alpha-1)(n-1)}{2n^2}(x_j-1)^2
$$
$$
-\sum_{1\le j<k\le n}\frac{\alpha-1}{n^2}(x_j-1)(x_k-1)+\sum_{|a_1+a_2+\ldots+a_n|=2}o\left(\prod_{j=1}^n(x_j-1)^{a_j}\right)
$$
$$
=1+\sum_{j=1}^n\frac{1}{n}(x_j-1)+\frac{\alpha-1}{2n^2}\left(n\sum_{j=1}^n(x_j-1)^2-\left(\sum_{j=1}^n(x_j-1)\right)^2\right)+o\left(\sum_{j=1}^n(x_j-1)^{2}\right).
$$
Using the similar formula for $
P_\beta$ and the fact that $\sum_{j=1}^n(x_j-1)=0$ we obtain
$$
P_\alpha-P_\beta=\frac{\alpha-\beta}{2}\sum_{j=1}^n(x_j-1)^2+o\left(\sum_{j=1}^n(x_j-1)^{2}\right).
$$
By expressing $P_\gamma-P_\delta$ similarly, we obtain that
$$
\lim_{(x_1,x_2,\ldots,x_n)\rightarrow (x,x,\ldots,x)} \frac{P_\alpha-P_\beta}{P_\gamma-P_\delta}=\frac{\alpha-\beta}{\gamma-\delta}.
$$
If, for example, $\alpha=0$, then to prove the limit in this case we can use the continuity of $
P_\alpha$ with respect to $\alpha$ at $\alpha=0$ near the point $(x,x,\ldots,x)$ $(x>0)$ and the fact that $\lim_{\alpha\rightarrow 0} P_\alpha=G_n,$ . Similarly for the cases $\beta=0$, $\gamma=0$, and $\delta=0$.
\end{proof}

In particular, if $\alpha=0$, $\beta=1$, $\gamma=\frac{1}{r}\ne1$, and $\delta=1$, then by Lemma 1,
$$
\lim_{(x_1,x_2,\ldots,x_n)\rightarrow (x,x,\ldots,x)} \frac{G_n-A_n}{P_\gamma-A_n}=\frac{r}{r-1}. \eqno(1)
$$

We will use an optimization argument similar to the one used in \cite{aliyev1} (p. 8), \cite{aliyev2}. A similar method was used before in \cite{wen1} (Lemma 2.1), \cite{wen2} (Lemma 1). Maclaurin and Cauchy used analogous optimization methods to prove AM-GM inequality (see \cite{hardy}, p. 19, footnote).

\begin{lemma} If $x,\ y,$ and $z$ are positive real numbers, $r\ne 0,1$ is a real number, $x\le y\le z$, $xyz=\alpha$, $x+y+z=\beta$, where $\alpha$ and $\beta$ are positive constants such that $\beta^3>27\alpha$, then as $y$ increases, function
$
h(x,y,z)=x^r+y^r+z^r
$
decreases when $r>1$,
and increases when $r<0$ or $0<r<1$. If $r>1$, then the maximum and the minimum of function $h(x,y,z)$ are achieved when $x=y<z$ and $x<y=z$, respectively.
If $r<0$ or $0<r<1$, then the maximum and the minimum of function $h(x,y,z)$ are achieved when $x<y=z$ and $x=y<z$, respectively.
\end{lemma}

\begin{proof} It was shown in \cite{aliyev1} (p. 8) that $x,y,z$ satisfying $xyz=\alpha$ and $x+y+z=\beta$ can be parametrized as
$$
x=\frac{-t+\beta\pm\sqrt{(t-\beta)^2-\frac{4\alpha}{t}}}{2},\ y=t,\ z=\frac{-t+\beta\mp\sqrt{(t-\beta)^2-\frac{4\alpha}{t}}}{2},
$$
where $t\in [t_1,t_2]$, and $t_1\in \left(0,\frac{\beta}{3}\right),$ $t_2\in \left(\frac{\beta}{3},\beta\right)$ are the zeros of the cubic function $\kappa(x)=t^3-2\beta t^2+\beta^2 t-4\alpha$. The zero $t_3>\beta$ of $\kappa(x)$ is not in $[t_1,t_2]$. The part of the curve satisfying $x\le y\le z$ is parametrized as
$$
x=\frac{-t+\beta-\sqrt{(t-\beta)^2-\frac{4\alpha}{t}}}{2},\ y=t,\ z=\frac{-t+\beta+\sqrt{(t-\beta)^2-\frac{4\alpha}{t}}}{2},
$$
where $t\in [t^*_1,t^*_2]$, and $t^*_1\in \left(t_1,\frac{\beta}{3}\right)$, $t^*_2\in \left(\frac{\beta}{3},t_2\right)$ are the zeros of the cubic function $\kappa^*(t)=-8t^3+4\beta t^2-4\alpha$ (see Figure 1). The zero $t^*_3$ of $\kappa^*(t)$ satisfies $t^*_3<0$, so, $t^*_3$ is not in $[t^*_1,t^*_2]$. Note that if $t=t^*_1$, then $x=y<z$, and if $t=t^*_2$, then $x<y=z$. For the remaining points $t$ in interval $\left(t^*_1,t^*_2\right)$, we have $x<y<z$.
\begin{figure}[htbp]
\centerline{\includegraphics[scale=.5]{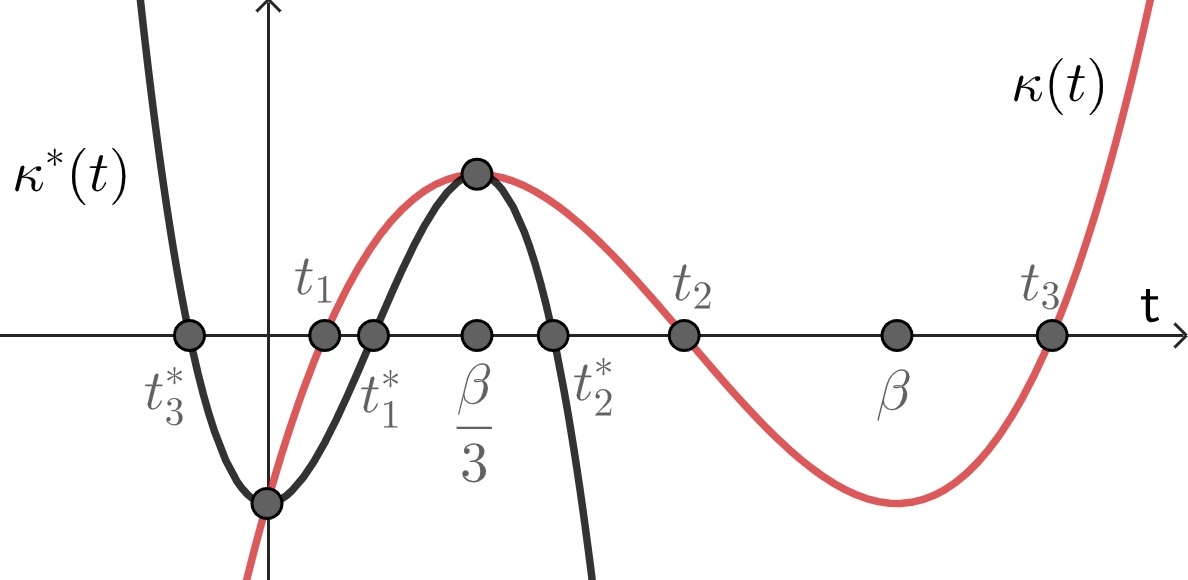}\ \includegraphics[scale=.5]{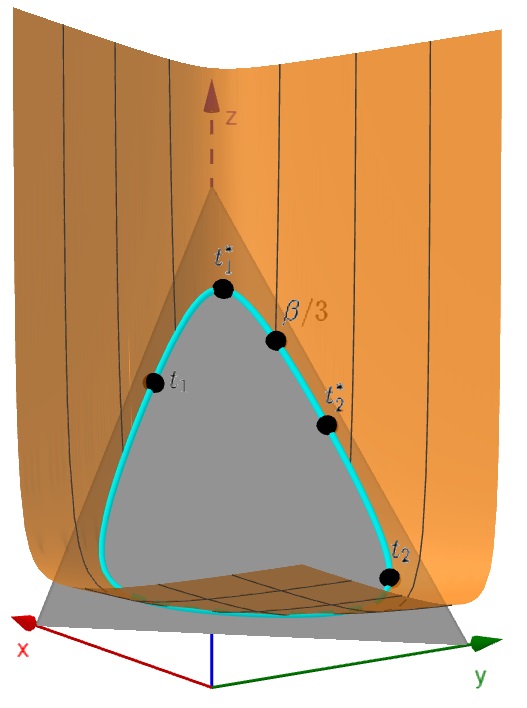}}
\label{fig1}
\caption{Left: The graphs of $\kappa(t)$ (red), $\kappa^*(t)$ (black) and their zeros $t_1$, $t_2$, $t_3$,  $t^*_1$, $t^*_2$, and $t^*_3$ (black). Right: Parametric space curve (blue) representing the intersection of plane $x+y+z=\beta$ (gray) and surface $xyz=\alpha$ (orange). The points corresponding to $t=\frac{\beta}{3}$, $t_1$, $t_2$,  $t^*_1$, and $t^*_2$ (black).}
\end{figure}
By Vieta's formulas $x$ and $z$ are the solutions of quadratic equation $w^2-(\beta-y)w+\frac{\alpha}{y}=0$. Using implicit differentiation we find that $w'_t=\frac{\frac{\alpha}{y^2}-w}{2w+y-\beta}$. So,  $x'_t=\frac{\frac{\alpha}{y^2}-x}{2x+y-\beta}$ and  $z'_t=\frac{\frac{\alpha}{y^2}-z}{2z+y-\beta}$. Then
$$
h'_t=rx^{r-1}\frac{\frac{\alpha}{y^2}-x}{2x+y-\beta}+ry^{r-1}+rz^{r-1}\frac{\frac{\alpha}{y^2}-z}{2z+y-\beta}.
$$
After simplifications we obtain that
$$
h'_t=\frac{r(y-x)(z-y)}{y(x-z)}\cdot \left(\frac{z^r-y^r}{z-y}-\frac{y^r-x^r}{y-x}\right).
$$
If $r<0$ or $r>1$, then function $t^r$ is concave up for $t>0$. Therefore $\frac{z^r-y^r}{z-y}>\frac{y^r-x^r}{y-x}$. Consequently, if $r>1$, then $h'_t<0$ ($h\downarrow$) and if $r<0$, then $h'_t>0$  ($h\uparrow$) . Similarly, if $0<r<1$, then function $t^r$ is concave down for $t>0$. Therefore $\frac{z^r-y^r}{z-y}<\frac{y^r-x^r}{y-x}$ (see Figure 2). Consequently, if $0<r<1$, then $h'_t>0$ ($h\uparrow$). This means that if $r>1$, then the maximum and the minimum of function $h(x,y,z)$ ($t\in [t^*_1,t^*_2]$) are achieved when $t^*_1=x=y<z=t^*_2$ and $t^*_1=x<y=z=t^*_2$, respectively. Similarly, if $r<0$ or $0<r<1$, then the maximum and the minimum of function $h(x,y,z)$ ($t\in [t^*_1,t^*_2]$) are achieved when $t^*_1=x<y=z=t^*_2$ and $t^*_1=x=y<z=t^*_2$, respectively. 
\end{proof}
\begin{figure}[htbp]
\centerline{\includegraphics[scale=.5]{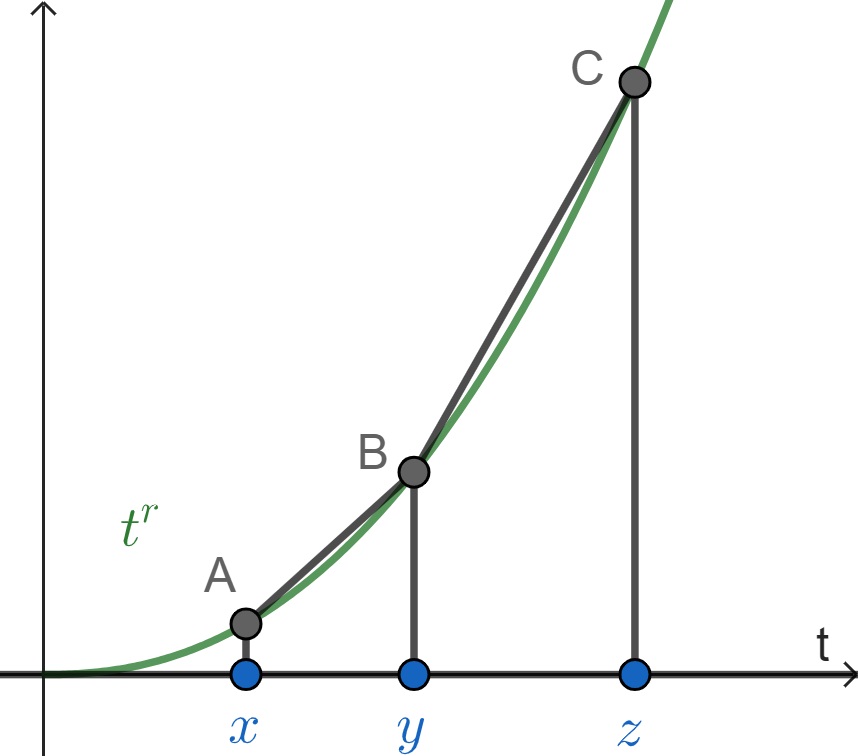}\ \includegraphics[scale=.5]{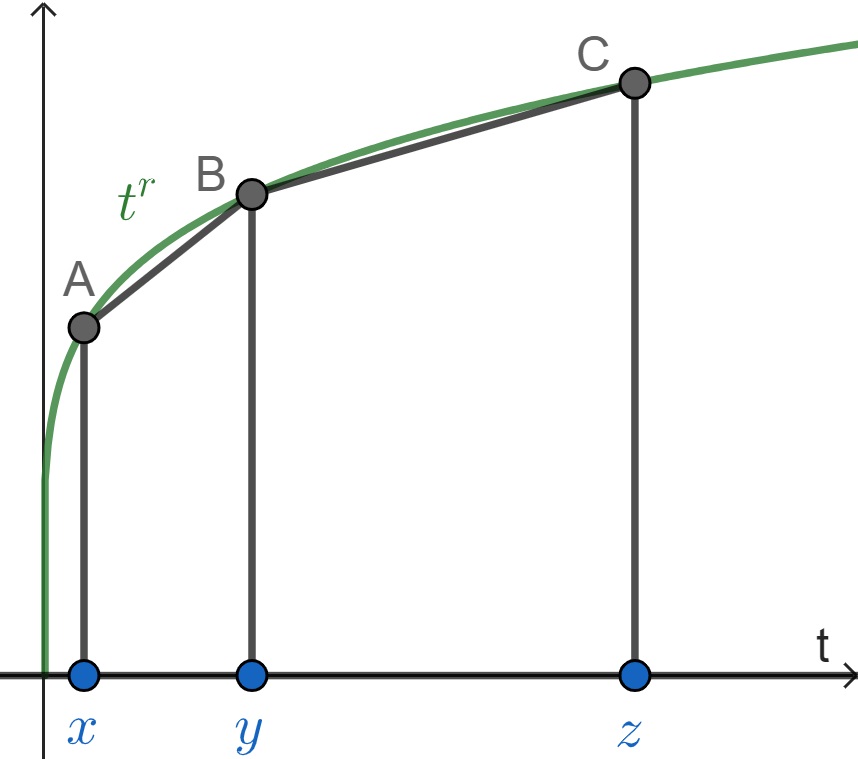}\ \includegraphics[scale=.5]{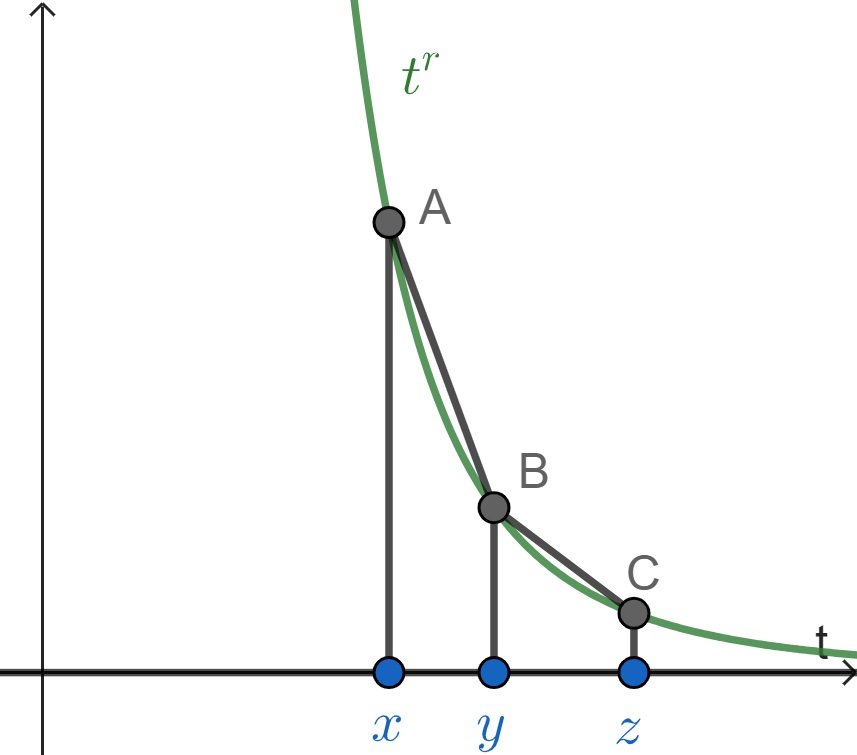}}
\label{fig2}
\caption{Comparison of slopes of chords $AB$ and $BC$ of function $t^r$ $(t>0)$ in cases $r>1$ (left), $0<r<1$ (center), and $r<0$ (right).}
\end{figure}

The following classical result will be essential for our study and can be proved easily based on equality
$$
\left(\frac{f_1}{f_2}\right)'=\left(\frac{f'_1}{f'_2}-\frac{f_1}{f_2}\right)\frac{f'_2}{f_2}.
$$

\begin{lemma} (\cite{hardy}, p. 106) If $f_1,\ f_2,$ and $f'_1/f'_2$ are positive increasing functions, then $f_1/f_2$ either increases for all $x$ in question, or decreases for all such $x$, or decreases to a minimum and then increases. In particular, if $f_1(0)=f_2(0)=0$, then $f_1/f_2$ increases for $x>0$.
\end{lemma}

For the purposes of clarity, we will also refer to the following more special result, which is sometimes referred to as "L'Hôpital type Monotonicity Rule" and was widely used (see e.g. \cite{wu}, \cite{qian}, \cite{tian}, \cite{he}, \cite{neu}, \cite{kouba}, \cite{bhayo2}) to study optimization problems similar to the one in the current paper.

\begin{lemma} (\cite{ander}, p. 10) Let $-\infty<a<+\infty$, and let $f_1,\ f_2:[a,b]\rightarrow \mathbb{R}$ be continuous on $[a,b]$ and differentiable on $(a,b)$. Let $f'_2(x)\ne0$ on $(a, b)$. Then, if $f'_1(x)/f'_2(x)$ is increasing (decreasing) on $(a , b)$, so are
$$[f_1(x) - f_1(a)] /[f_2(x) - f_2(a)] \text{ and } [f_1(x) - f_1(b)] /[f_2(x) - f_2(b)].$$
If $f'_1(x)/f'_2(x)$ is strictly monotone, then the monotoneity in the conclusion is also strict.
\end{lemma}

As was pointed out in \cite{ander1} (p. 805) results similar to Lemma 2.4 appeared before in \cite{cheag} (p. 42) and \cite{ander0} (p. 6-7) (see also \cite{pin}, \cite{pin2}, \cite{pin3}, \cite{alzer}). 

\section{Main results}

In the following we will focus on the study of maximum and minimum values of $\frac{A_n-G_n}{P_\alpha-G_n}$. This question is meaningful if $\alpha\ne0,1$. Therefore we will assume that $\alpha=\frac{1}{r}$, $r\ne0,1$ and write just $P$ instead of $P_\alpha$. Also, without loss of generality we can assume that $\sum\limits_{i =
1}^n{x_{i}}=1$. The case $n=2$ follows from the results in \cite{wu}, \cite{kouba}, \cite{james} so in the following we assume that $n>2$.

Let us denote
$$
q(x_1,\ldots,x_n)=\begin{cases}

{\frac{r}{r-1}}, & \text{if } x_i=\frac{1}{n}\ (i=1,2,\ldots,n);\\
{\frac{G_n-\frac{1}{n}}{P-\frac{1}{n}}}, & \text{if otherwise.}

\end{cases}
$$
By Power mean inequality (see e.g. \cite{beck}, p. 17) if $0<r<1$, then $P\ge\frac{1}{n}$, and if $r>1$ or $r<0$, then $P\le\frac{1}{n}$. Also, $\frac{1}{n}\ge G_n$ and the equalities are possible if and only if all $x_i$ ($i=1,2,\ldots,n$) are equal to $\frac{1}{n}$. Similarly,  if $r>0$, then $P\ge G_n$, and if $r<0$, then $P\le G_n$.
By Equation (1), function $q(x_1,\ldots,x_n)$ is continuous at point $\left(\frac{1}{n},\frac{1}{n},\ldots,\frac{1}{n}\right)$. So, function $q(x_1,\ldots,x_n)$ is non-zero and continuous over compact set
$$
\mathbb{D}=\{(x_{1},x_{2},\ldots ,x_{n})|\ x_{1},x_{2},\ldots ,x_{n}\ge 0;\sum\limits_{i =
1}^n{x_{i}}=1\},
$$
and
$\frac{A_n-G_n}{P_\alpha-G_n}=-\frac{1}{\frac{1}{{q(x_1,\ldots,x_n)}}-1}$, whenever the functions on the left hand side and the right hand side of this equality are defined. By classical result in analysis (see e.g. \cite{rudin}, Theorem 4.16) a real-valued function continuous over a compact metric space attains its extremal values on this space. So, there is a point $(x_1,\ldots,x_n)\in \mathbb{D}$ where $q$ attains its maximum (minimum). For the point $(x_1,\ldots,x_n)$ we can assume without loss of generality that $x_1\le x_2\le\ldots\le x_n$. We will consider 2 cases $x_1>0$ and $x_1=0$.

\noindent \textbf{The case $x_1>0$.}
Take any 3 consecutive coordinates of point $(x_1,\ldots,x_n)$, say $x=x_i,\ y=x_{i+1},$ and $z=x_{i+2}$. By Lemma 2.2, either $x=y\le z$ or $x\le y=z$. This means that there can be at most 2 different numbers in set $\{x_1,\ldots,x_n\}$. In other words, either (i) $x_{1}=x_{2}=\ldots =x_{n-1}\le x_{n}$ or (ii) $x_{1}\le x_{2}=\ldots =x_{n-1}=x_{n}$. We can unify these 2 cases into one by taking $$x_{1}=x_{2}=\ldots =x_{n-1}=x,\ x_{n}=1-(n-1)x,$$ where $0\le x\le\frac{1}{n-1}$. Indeed, the cases (i) and (ii), with possible renaming of the variables, correspond to intervals $0\le x\le\frac{1}{n}$ and $\frac{1}{n}\le x\le\frac{1}{n-1}$, respectively. Thus our problem is reduced to study of function $f(x)=\frac{g(x)-\frac{1}{n}}{p(x)-\frac{1}{n}}$ in interval $0\le x\le\frac{1}{n-1}$, where $$g(x)=\sqrt[n]{x^{n-1}(1-(n-1)x)},\ p(x)=\left(\frac{(n-1)x^{\frac{1}{r}}+(1-(n-1)x)^{\frac{1}{r}}}{n}\right)^r.$$
Note that $g(0)=g\left(\frac{1}{n-1}\right)=0$. If $r>0$, then $p(0)=\frac{1}{n^r}$, and $p\left(\frac{1}{n-1}\right)=\frac{(n-1)^{r-1}}{n^r}$, otherwise $p(0)=p\left(\frac{1}{n-1}\right)=0$. Therefore, if $r>0$, then $f(0)=\frac{n^{r-1}}{n^{r-1}-1}$, and $f\left(\frac{1}{n-1}\right)=\frac{n^{r-1}}{n^{r-1}-(n-1)^{r-1}}$, otherwise $f(0)=f\left(\frac{1}{n-1}\right)=1$.
We also calculate
$$g'(x)= \frac{(n-1)(1-nx)}{nx(1-(n-1)x)}g(x),$$
$$p'(x)=\frac{(n-1)\left(x^{\frac{1}{r}-1}-(1-(n-1)x)^{\frac{1}{r}-1}\right)}{(n-1)x^{\frac{1}{r}}+(1-(n-1)x)^{\frac{1}{r}}}p(x).$$
In particular, from these formula it follows that if $0< x<\frac{1}{n}$ ($\frac{1}{n}< x<\frac{1}{n-1}$), then $g(x)\uparrow$ ($g(x)\downarrow$), and $0\le g(x)\le \frac{1}{n}$. Also $g'(0)=+\infty$ and $g'\left(\frac{1}{n-1}\right)=-\infty$. Similarly, if $r<0$ or $r>1$, then for interval $0< x<\frac{1}{n}$ ($\frac{1}{n}< x<\frac{1}{n-1}$), function $p(x)\uparrow$ ($p(x)\downarrow$). If $0<r<1$, then for interval $0< x<\frac{1}{n}$ ($\frac{1}{n}< x<\frac{1}{n-1}$), function $p(x)\downarrow$ ($p(x)\uparrow$). By considering maximal minimal values of $p(x)$ and $g(x)$, one can see that $f(x)$ doesn't change its sign in interval $0< x<\frac{1}{n-1}$. We can determine also the values of $p'(x)$ at points $x=0$ and $x=\frac{1}{n-1}$ (see Table 1).

\begin{table}
\begin{tabular}{ |c||c|c|c|  }

 \hline
 & $r<0$ & $0<r<1$& $1<r$\\
 \hline
 $p'(0)$   & $\left(\frac{n-1}{n}\right)^r$  & $-\frac{n-1}{n^r}$  & $+\infty$\\
 $p'\left(\frac{1}{n-1}\right)$ & $-\frac{n-1}{n^r}$ & $\left(\frac{n-1}{n}\right)^r$  &   $-\infty$\\
 
\hline
\end{tabular}

\label{tab1}
\caption{The values of $p'(x)$ at points $0$ and $\frac{1}{n-1}$.}
\end{table}
We can also write
$$
\frac{g'(x)}{p'(x)}=\frac{g(x)}{p(x)}\cdot \frac{x-\frac{1}{n}}{x-\frac{1}{n-1+\left(\frac{1}{x}-(n-1)\right)^{\frac{1}{r}}}}.
$$
Similarly,
$$g''(x)= -\frac{(n-1)g(x)}{n^2x^2(1-(n-1)x)^2}= -\frac{g'(x)}{nx(1-(n-1)x)(1-nx)},$$
$$p''(x)=-\frac{(n-1)(r-1)x^{\frac{1}{r}-2}(1-(n-1)x)^{\frac{1}{r}-2}p(x)}{r\left((n-1)x^{\frac{1}{r}}+(1-(n-1)x)^{\frac{1}{r}}\right)^2}$$
$$=-\frac{(r-1)x^{\frac{1}{r}-2}(1-(n-1)x)^{\frac{1}{r}-2}p'(x)}{r\left((n-1)x^{\frac{1}{r}}+(1-(n-1)x)^{\frac{1}{r}}\right)\left(x^{\frac{1}{r}-1}-(1-(n-1)x)^{\frac{1}{r}-1}\right)}.$$
 \newcommand{\sgn}{\mathop{\mathrm{sgn}}}Therefore, $g''(x)<0$ and $\sgn{p''(x)}=-\sgn{\frac{r}{r-1}}$ in interval $0< x<\frac{1}{n-1}$.
We can also write
$$
\frac{g''(x)}{p''(x)}=\frac{g'(x)}{p'(x)}\cdot\frac{r}{n(r-1)}\cdot \frac{1-(n-1)x}{1-nx}\left(\frac{n-1}{\left(\frac{1-(n-1)x}{x}\right)^{\frac{1}{r}}}+1\right)\left(1-\left(\frac{1-(n-1)x}{x}\right)^{\frac{1}{r}-1}\right).
$$
If we denote $s(x)=\frac{x}{1-(n-1)x}$ and $W(x)=U(x)\cdot V(x)$, where $U(x)= \frac{(s(x))^{1-\frac{1}{r}}-1}{\left({1-\frac{1}{r}}\right)\left(s(x)-1\right)}$ and $V(x)=\frac{(n-1)(s(x))^{\frac{1}{r}}+1}{n}$, then this can be written simply as
$$
\frac{g''(x)}{p''(x)}=\frac{g'(x)}{p'(x)}\cdot W(x).
$$
Note that if $0< x<\frac{1}{n-1}$, then $s(x)\uparrow$ and $0<s(x)<+\infty$. Also $s(0)=0$, $s(\frac{1}{n})=1$, and $\lim_{x\rightarrow\frac{1}{n-1}}s(x)=+\infty$.

\noindent\textit{Monotonicity of $U(x)$ and $V(x)$.} First of all note that $U(\frac{1}{n})=1$, and $V(\frac{1}{n})=1$.
Obviously, in interval $0\le x\le\frac{1}{n-1}$, if $r>0$, then function $V(x)\uparrow$, and if $r<0$, then function $V(x)\downarrow$. Similarly, in interval $0\le x\le\frac{1}{n-1}$, if $0<r<1$ or $r>1$, then function $U(x)\downarrow$, and if $r<0$, then function $U(x)\uparrow$. In interval $\frac{1}{n}\le x\le\frac{1}{n-1}$ this can be proved by using Lemma 2.4 for functions $f_1= \frac{(s(x))^{1-\frac{1}{r}}}{\left({1-\frac{1}{r}}\right)}$ and $f_2=s(x)$, with $a=\frac{1}{n}$. In interval $0\le x\le\frac{1}{n}$ this can be proved by using Lemma 2.4 for the same functions $f_1$ and $f_2$, with $b=\frac{1}{n}$. Note that $\frac{f'_1(x)}{f'_2(x)}=(s(x))^{-\frac{1}{r}}$ is decreasing if $r>0$ and increasing if $r<0$.

\noindent\textit{Monotonicity intervals of $W(x)$.} We calculate
$$
W'(x)=\frac{\left(n-1\right) s^{\frac{-r+1}{r}}-s^{\frac{r-1}{r}}+\left(-r+1\right) s^{-\frac{1}{r}}+\left(n-1\right)\left( r-1\right) s^{\frac{1}{r}}-\left(n-2\right) r}{\left(s-1\right)^{2} \left(r-1\right) n}\cdot s'(x),
$$
where $s=s(x)$. We can check that $\frac{s'(x)}{\left(s-1\right)^{2}}=\frac{1}{(nx-1)^2}$. Therefore the values of $
W(x)$ and $
W'(x)$ at points $x=0$ and  $x=\frac{1}{n-1}$ can be calculated (see Table 2). We also find $W\left(\frac{1}{n}\right)=1$, $W'\left(\frac{1}{n}\right)=\frac{n(n-2)}{2r}$ and if $r>1$, then $W(0)=\frac{r}{n(r-1)}$ and $W\left(\frac{1}{n-1}\right)=\frac{r(n-1)}{n(r-1)}$.

\noindent
\begin{table}
\begin{tabular}{ |c||c|c|p{2cm}|p{2cm}|p{1.1cm}|p{1.1cm}|  }

 \hline
 & $r<0$ & $0<r<1$& $1<r<2$, $n<\frac{r}{r-1}$  & $1<r<2$, $n\ge\frac{r}{r-1}$ & $2\le r$, $n\ge r$ &  $2\le r$, $n< r$ \\
 \hline
 $W(0)$   & $+\infty$  & $+\infty$  & $>1$    &$\le1$&   $<1$&   $<1$\\
 $W\left(\frac{1}{n-1}\right)$ & $+\infty$  & $+\infty$  &   $>1$   & $>1$    &$\ge1$&$<1$ \\
 
\hline
\end{tabular}

\noindent

\begin{tabular}{ |c||c|c|c|c|c|  }

 \hline
 & $r<0$ & $0<r<1$  & $1<r<2$ & $r=2$ & $2<r$\\
 \hline

 $W'\left(0\right)$ & $-\infty$  & $-\infty$  &   $-\infty$  & $+\infty$   &$+\infty$\\
 $W'\left(\frac{1}{n-1}\right)$& $+\infty$  & $+\infty$   &$+\infty$ & $+\infty$&  $-\infty$\\
\hline

\end{tabular}
\label{tab2}
\caption{The values of $W(x)$ and $W'(x)$ at points $0$ and $\frac{1}{n-1}$.}
\end{table}
If $0<r<1$; $1<r$, $n\le\frac{r}{r-1}$; or $1<r<2$, $n\ge\frac{r}{r-1}$, then we can take $f_1=V(x)$, $f_2=\frac{1}{U(x)}$ in Lemma 2.3, so that $W(x)=\frac{f_1}{f_2}$, and use the fact that
$$
\frac{f_1^{\prime}}{f_2^{\prime}}=\frac{ r\left(n-1\right) s^{\frac{1}{r}} \left(s^{\frac{r-1}{r}}-1\right)^{2}}{n \left(r-1\right) \left(s^{\frac{r-1}{r}} r+s^{\frac{2 r-1}{r}}-r s-s^{\frac{r-1}{r}}\right)},
$$
is increasing for $0<r<2$ ($r\ne1$). Indeed, we can write this function as
$$
\frac{f_1^{\prime}}{f_2^{\prime}}=
\frac{\alpha \left(n-1\right) s^{\alpha-1} \left(s^{1-\alpha}-1\right)^{2}}{n\left(1-\alpha\right)  \left(\alpha \,s^{1-\alpha}-1+\left(1-\alpha\right) s^{-\alpha}\right)}.
$$
Then
$$
\left(\frac{f_1^{\prime}}{f_2^{\prime}}\right)_s'=\frac{ \alpha s^{\alpha -1} \left(n-1\right)h(s) \left(s^{1-\alpha}-1\right)}{ ns \left(\alpha  s^{1-\alpha}-s^{-\alpha} \alpha +s^{-\alpha}-1\right)^{2}}
,
$$
where $$h(s)=s^{1-2\alpha}+2 \alpha  s^{1-\alpha}-2 s^{-\alpha} \alpha -s^{1-\alpha}+s^{-\alpha}-1.$$Note that 
$$h'(s)=-\frac{2 \alpha-1}{s^{\alpha+1}}\cdot\left(\alpha  s- \alpha +s^{1- \alpha}-s\right).$$
By Bernoulli's inequality if $\frac{1}{2}<\alpha<1$, then $$s^{1- \alpha}=(1+(s-1))^{1- \alpha}\le1+(1-\alpha)(s-1)=\alpha-\alpha  s +s,$$and therefore $h(s)$ is increasing for $0<s<+\infty$. Since $h(1)=0$, it follows that $h(s)<0$ for $0<s<1$ and  $h(s)>0$ for $s>1$. By noting that for $\frac{1}{2}<\alpha<1$, the function $s^{1-\alpha}-1$ has the same sign change near the point $s=1$, we obtain that  $
\left(\frac{f_1^{\prime}}{f_2^{\prime}}\right)_s'>0$ both for $0<s<1$ and for $s>1$. Consequently, $
\frac{f_1^{\prime}}{f_2^{\prime}}$ is an increasing function. The case $\alpha >1$ is considered similarly.

By Lemma 2.3 and the details about function $W(x)$ in Table 1, it decreases to a minimum and then increases. If $r<0$, then we can change the variable $x\longrightarrow\frac{1}{n-1}-x$ in $f_1=V(x)$, $f_2=\frac{1}{U(x)}$, and show again using Lemma 2.3, that $W(x)$ decreases to a minimum and then increases. Finally, if $2<r$, $n\ge r$; or $2<r$, $n< r$, then we can take $f_1=\frac{1}{U(x)}$ and $f_2=V(x)$, so that $\frac{1}{W(x)}=\frac{f_1}{f_2}$, and $
\frac{f_1^{\prime}}{f_2^{\prime}}$
is increasing. Consequently, $\frac{1}{W(x)}$ decreases to a minimum and then increases, which means that ${W(x)}$ increases to a maximum and then decreases. If $r=2$, then
$$
W'(x)=\frac{\left(n-2\right) \left(s^{\frac{1}{2}}+s^{-\frac{1}{2}}-2\right) s'(x)}{\left(s-1\right)^{2} n}>0,
$$
and therefore $W(x)$ is increasing in interval $\left(0, \frac{1}{n-1}\right)$. The resulting graphs are shown in Figure 3.
\begin{figure}[htbp]
\centerline{\includegraphics[scale=.4]{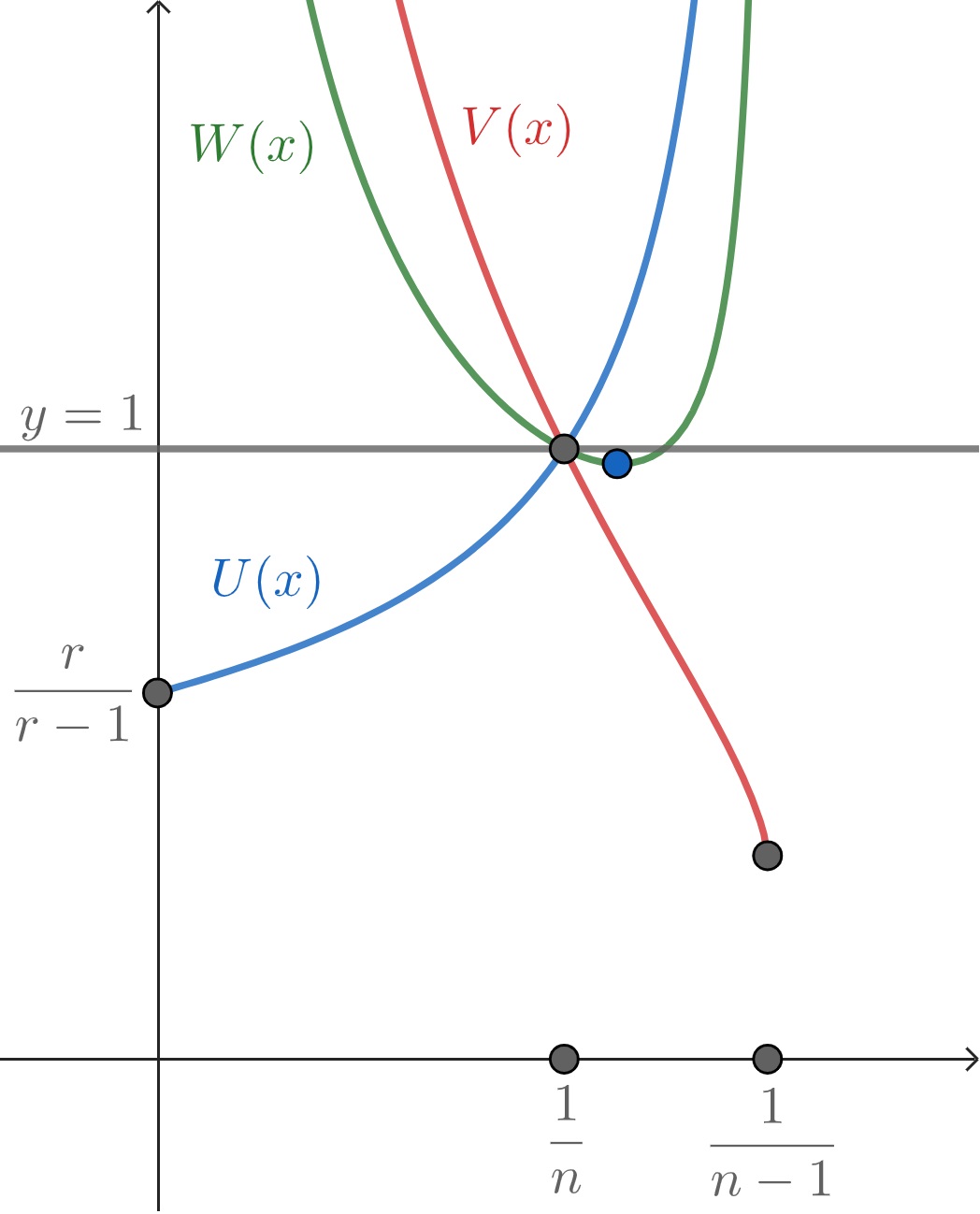}\ \includegraphics[scale=.4]{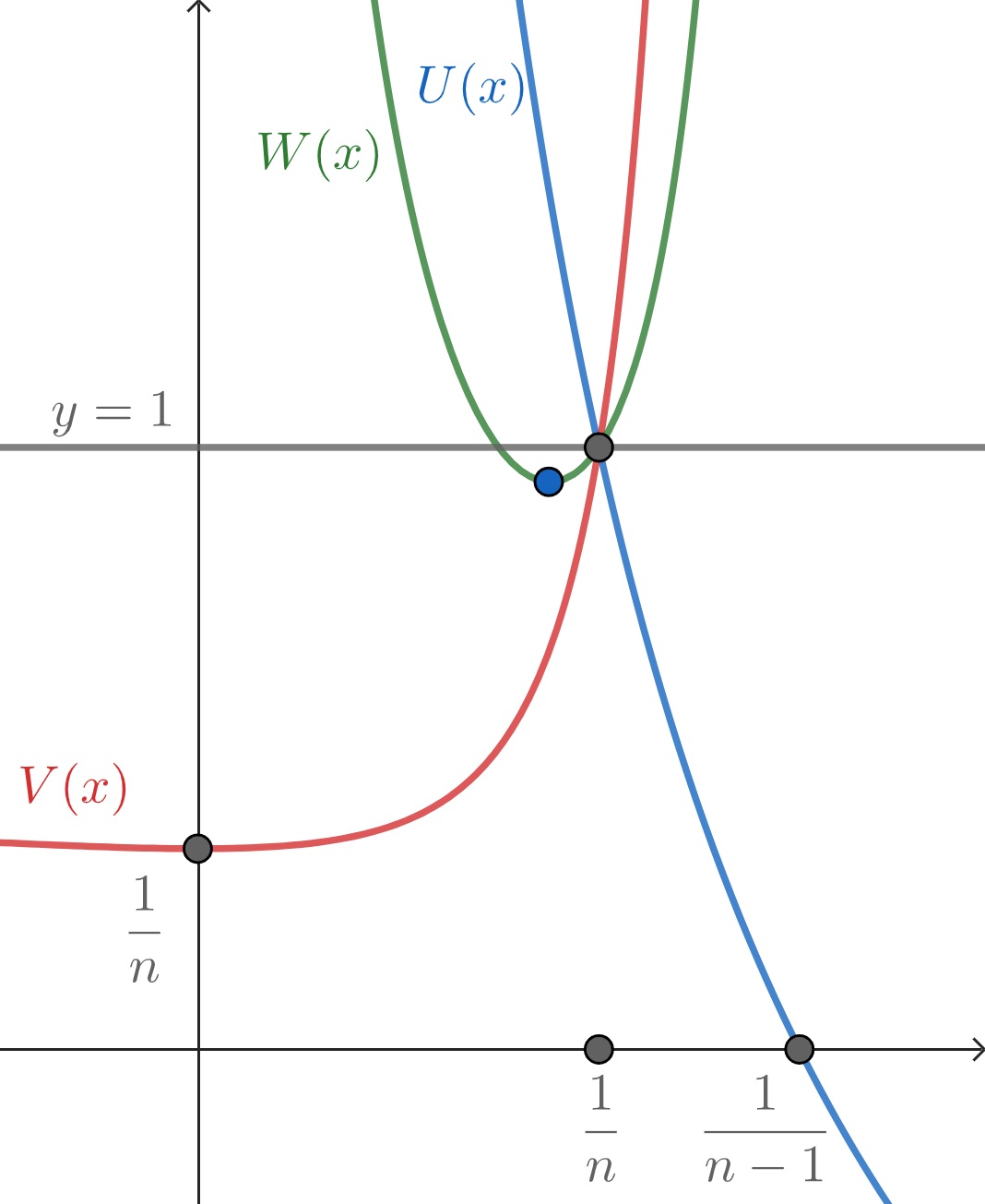}}
\centerline{\includegraphics[scale=.4]{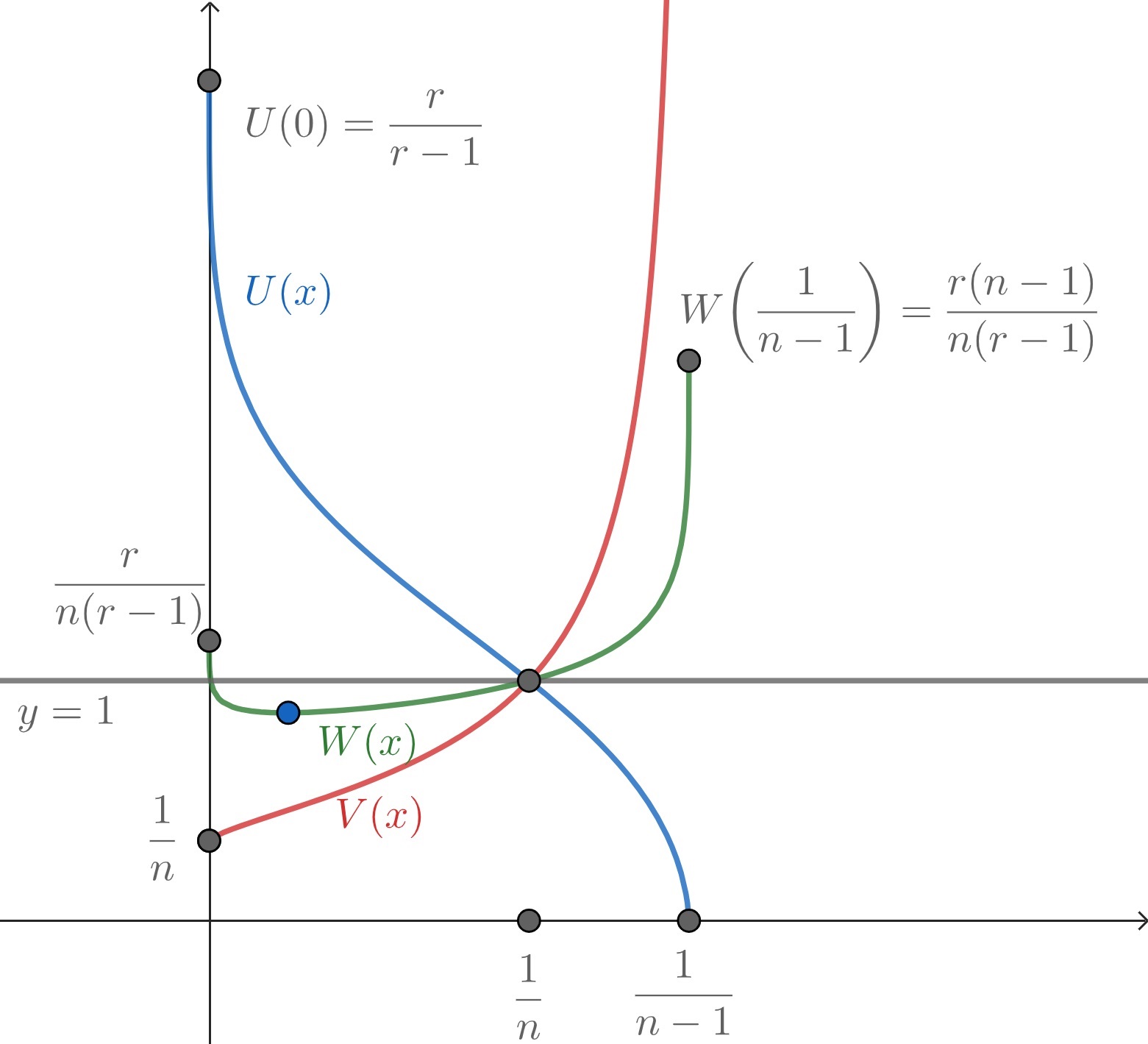}\ \includegraphics[scale=.4]{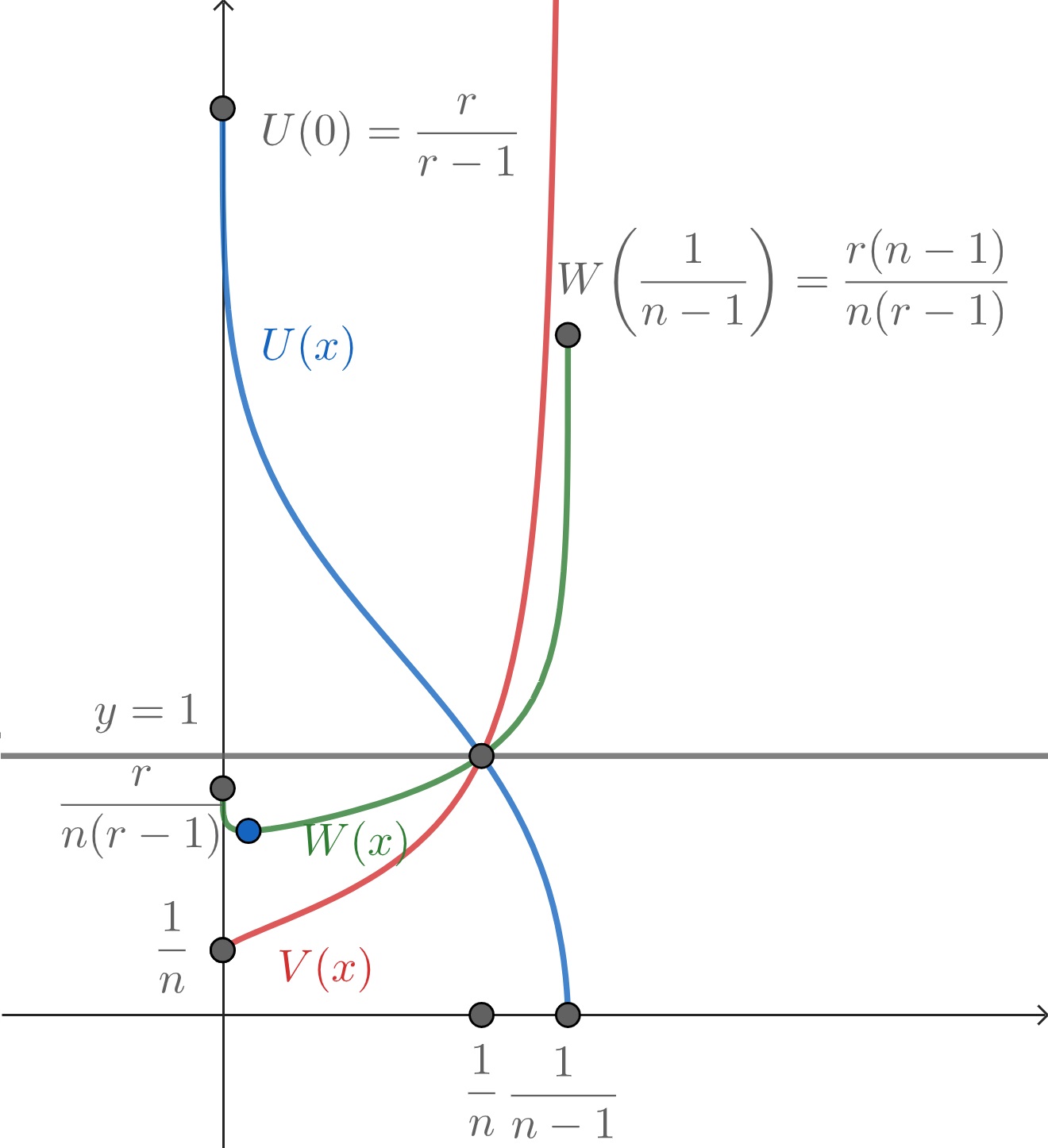}}
\centerline{\includegraphics[scale=.4]{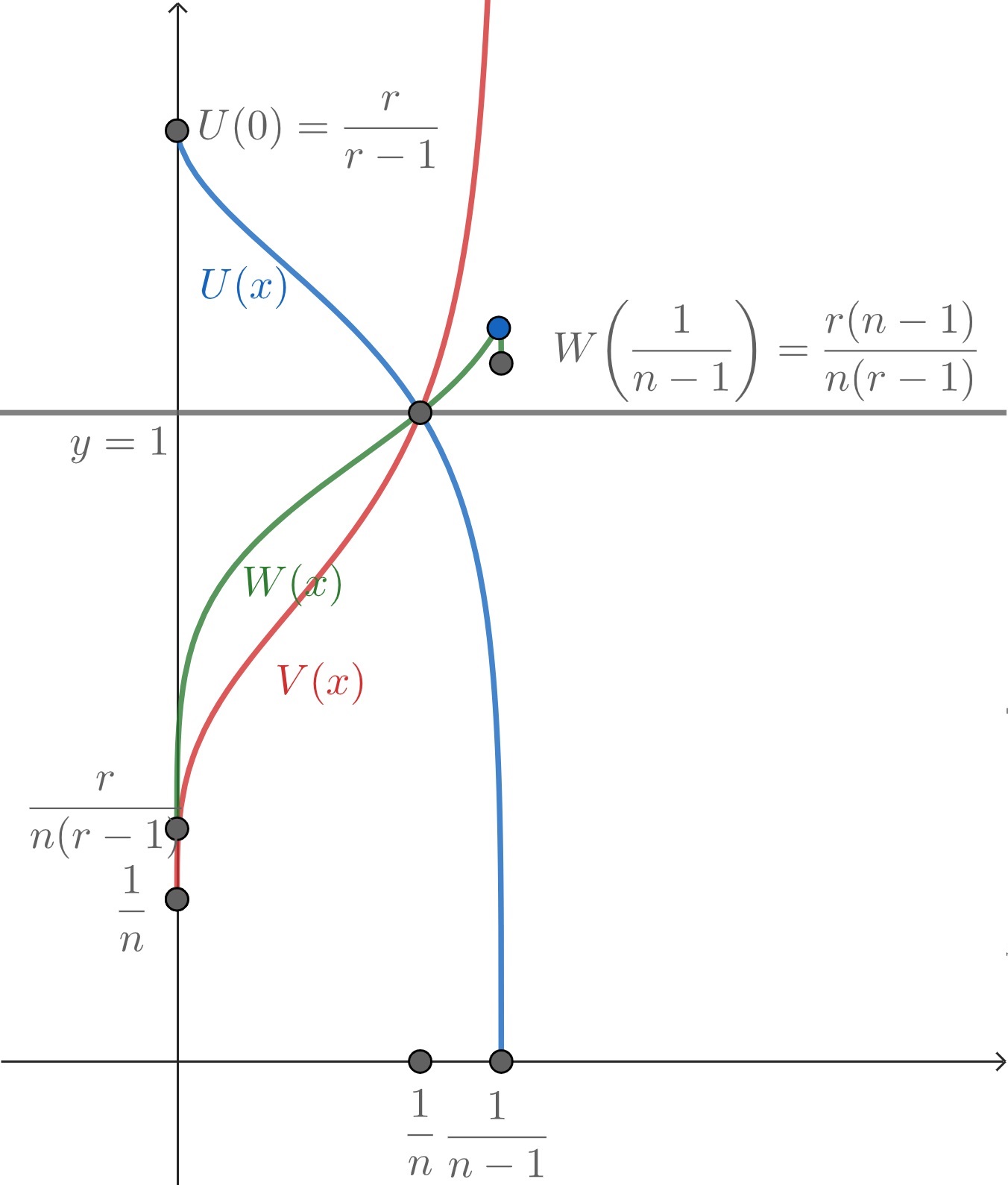}\ \includegraphics[scale=.4]{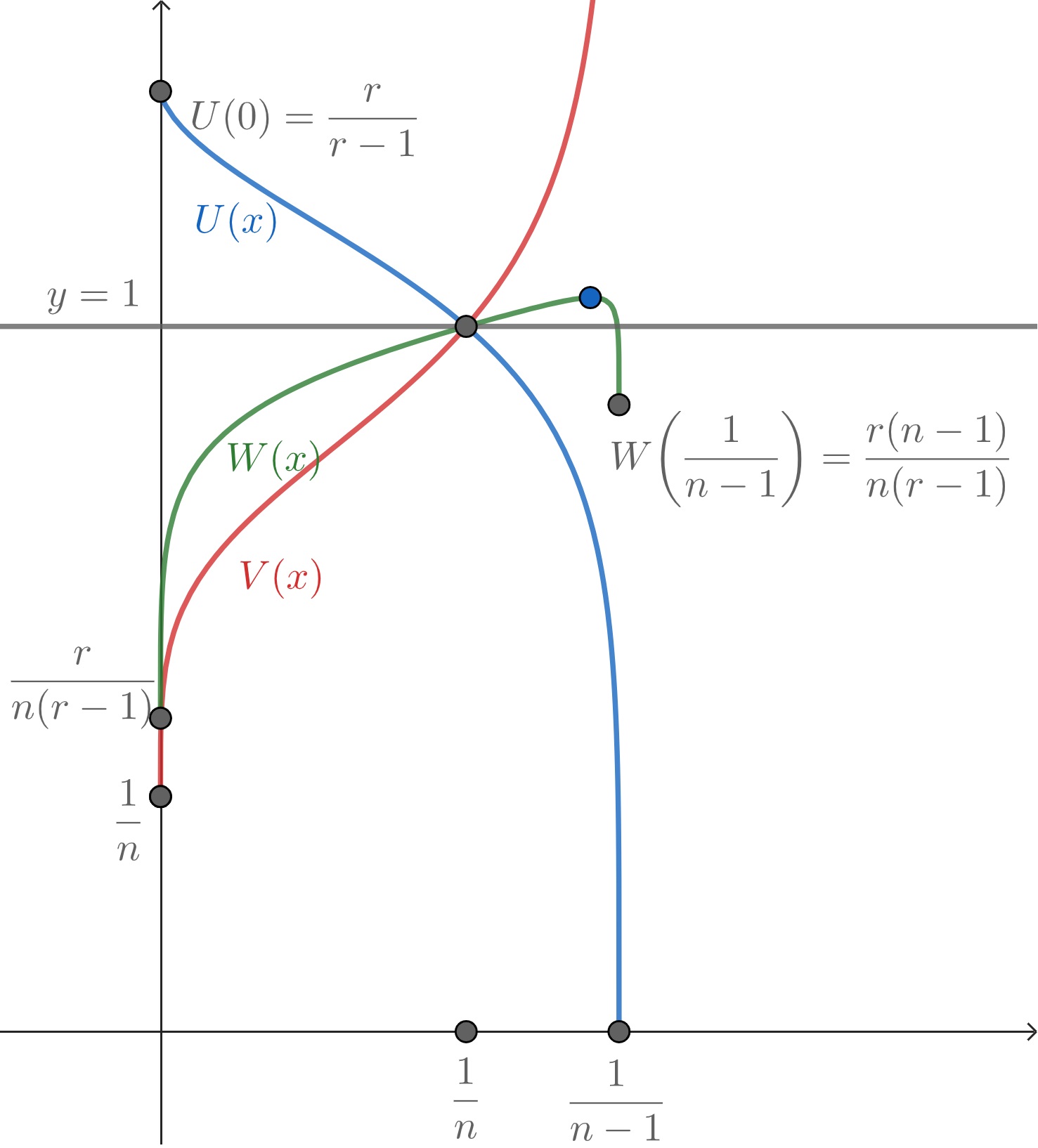}}
\label{fig3}
\caption{The graphs of $y=1$ (black), $U(x)$ (blue), $V(x)$ (red), and $W(x)$ (green) in interval $\left(0,\frac{1}{n-1}\right)$. Top left: $r<0$; Top right: $0<r<1$;\\
Middle left: $1<r<2$ and $n<\frac{r}{r-1}$; Middle right: $1<r\le2$ and $n\ge\frac{r}{r-1}$;\\
Bottom left: $2\le r$ and $n\ge r$; Bottom right: $2<r$ and $n<r$.\\
The blue point represents the extremum point of $W(x)$.}
\end{figure}

\noindent\textit{Monotonicity intervals of $\frac{g'(x)}{p'(x)}$.} We can now write
$$
\left(\frac{g'(x)}{p'(x)}\right)'=\left(\frac{g''(x)}{p''(x)}-\frac{g'(x)}{p'(x)}\right)\frac{p''(x)}{p'(x)}=\left(W(x)-1\right)\frac{p''(x)g'(x)}{(p'(x))^2}.
$$
We will consider each case separately. In each case we will use the sign of $p''(x)$ and the change of sign of $g'(x)$ at point $x=\frac{1}{n}$.
\begin{itemize}
\item If $r<0$, then $W(x)>1$ in intervals $\left(0,\frac{1}{n}\right)$ and $\left(\mu_1, \frac{1}{n-1}\right)$, and $W(x)<1$ in interval $\left(\frac{1}{n},\mu_1\right)$, where $\mu_1$ is point in interval $\left(\frac{1}{n},\frac{1}{n-1}\right)$, where $W(x)=1$ (see Figure 3, top left). Therefore function $\frac{g'(x)}{p'(x)}$ is decreasing in interval $\left(0,\mu_1\right)$ and increasing in $\left(\mu_1, \frac{1}{n-1}\right)$.

\item If $0<r<1$, then $W(x)>1$ in intervals $\left(0,\mu_2\right)$ and $\left(\frac{1}{n}, \frac{1}{n-1}\right)$, and $W(x)<1$ in interval $\left(\mu_2,\frac{1}{n}\right)$, where $\mu_2$ is point in interval $\left(0,\frac{1}{n}\right)$, where $W(x)=1$ (see Figure 3, top right). Therefore function $\frac{g'(x)}{p'(x)}$ is increasing in interval $\left(0,\mu_2\right)$ and decreasing in $\left(\mu_2, \frac{1}{n-1}\right)$.

\item If $1<r<2$ and $n<\frac{r}{r-1}$, then $W(x)>1$ in intervals $\left(0,\mu_3\right)$ and $\left(\frac{1}{n}, \frac{1}{n-1}\right)$, and $W(x)<1$ in interval $\left(\mu_3,\frac{1}{n}\right)$, where $\mu_3$ is point in interval $\left(0,\frac{1}{n}\right)$, where $W(x)=1$ (see Figure 3, middle left). Therefore function $\frac{g'(x)}{p'(x)}$ is decreasing in interval $\left(0,\mu_3\right)$ and increasing in $\left(\mu_3, \frac{1}{n-1}\right)$.

\item If $1<r\le2$ and $n\ge\frac{r}{r-1}$, then $W(x)<1$ in interval $\left(0,\frac{1}{n}\right)$, and $W(x)>1$ in interval $\left(\frac{1}{n},\frac{1}{n-1}\right)$ (see Figure 3, middle right). Therefore function $\frac{g'(x)}{p'(x)}$ is increasing in $\left(0, \frac{1}{n-1}\right)$.

\item If $2\le r$ and $n\ge r$, then $W(x)<1$ in intervals $\left(0,\frac{1}{n}\right)$, and $W(x)>1$ in interval $\left(\frac{1}{n},\frac{1}{n-1}\right)$ (see Figure 3, bottom left). Therefore function $\frac{g'(x)}{p'(x)}$ is increasing in $\left(0, \frac{1}{n-1}\right)$.

\item If $2<r$ and $n<r$, then $W(x)<1$ in intervals $\left(0,\frac{1}{n}\right)$ and $\left(\mu_4, \frac{1}{n-1}\right)$, and $W(x)>1$ in interval $\left(\frac{1}{n},\mu_4\right)$, where $\mu_4$ is point in interval $\left(\frac{1}{n},\frac{1}{n-1}\right)$, where $W(x)=1$ (see Figure 3, bottom right). Therefore function $\frac{g'(x)}{p'(x)}$ is increasing in interval $\left(0,\mu_4\right)$ and decreasing in $\left(\mu_4, \frac{1}{n-1}\right)$.

\end{itemize}

\noindent\textit{Monotonicity intervals of $f(x)$.}
We can now return to the study of function $f(x)=\frac{g(x)-\frac{1}{n}}{p(x)-\frac{1}{n}}$ in interval $0\le x\le\frac{1}{n-1}$. We calculate
$$
f'(x)=\left(\frac{g(x)-\frac{1}{n}}{p(x)-\frac{1}{n}}\right)'=\left(\frac{g'(x)}{p'(x)}-\frac{g(x)-\frac{1}{n}}{p(x)-\frac{1}{n}}\right)\frac{p'(x)}{p(x)-\frac{1}{n}}.
$$

\noindent
\begin{table}
\begin{tabular}{ |c||c|c|p{2cm}|p{2cm}|p{1.5cm}|p{1.5cm}|  }

 \hline
 & $r<0$ & $0<r<1$& $1<r<2$, $n<\frac{r}{r-1}$  & $1<r\le2$, $n\ge\frac{r}{r-1}$ & $2\le r$, $n\ge r$ &  $2<r$, $n< r$ \\
 \hline
 $f'(0)$   & $-\infty$  & $+\infty$  & $-\infty$    &$+\infty$&   $+\infty$&   $+\infty$\\
 $f'\left(\frac{1}{n-1}\right)$ & $+\infty$  & $-\infty$  &   $+\infty$   & $+\infty$    &$+\infty$&$-\infty$ \\
 
\hline
\end{tabular}

\label{tab3}
\caption{The values of $f'(x)$ at points $0$ and $\frac{1}{n-1}$.}
\end{table}

Again, we will consider each case separately. In each case we use Lemma 2.4 with $b=\frac{1}{n}$ in interval $\left(0, \frac{1}{n}\right)$ or its smaller part $\left(\mu_i, \frac{1}{n}\right)$ for $i=2,3$, and $a=\frac{1}{n}$ in interval $\left(\frac{1}{n},\frac{1}{n-1}\right)$ or its smaller part $\left(\frac{1}{n},\mu_i\right)$ for $i=1,4$. In the remaining intervals, if there are any, we use Lemma 2.3. For a proper use of Lemma 2.3, it is necessary to know the sign of $f'(x)$ at endpoints $x=0$ and $x=\frac{1}{n-1}$ (see Table 3).

\begin{figure}[htbp]
\centerline{\includegraphics[scale=.4]{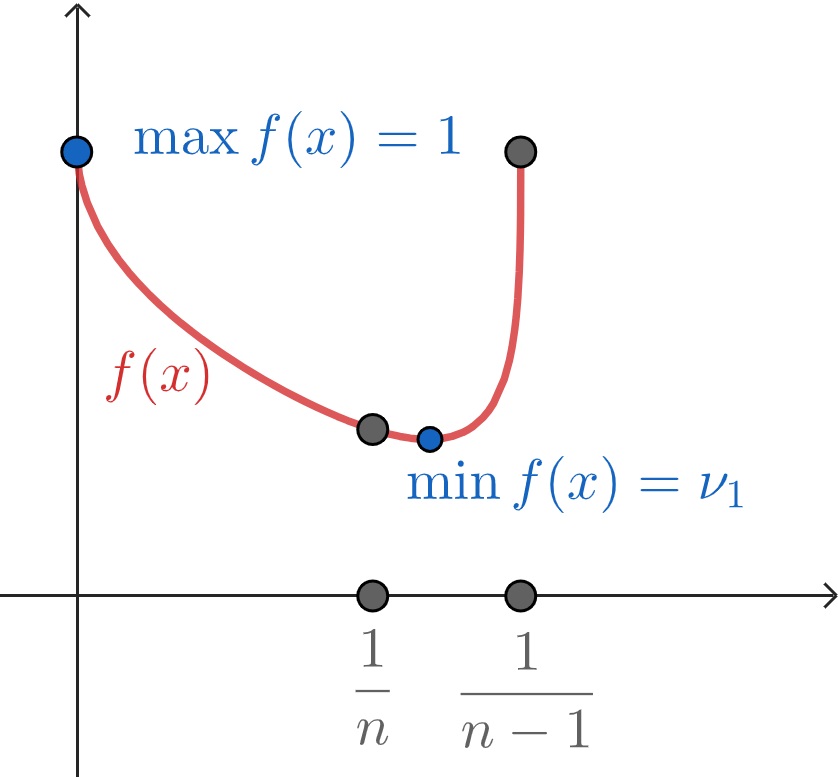}}
\centerline{\includegraphics[scale=.4]{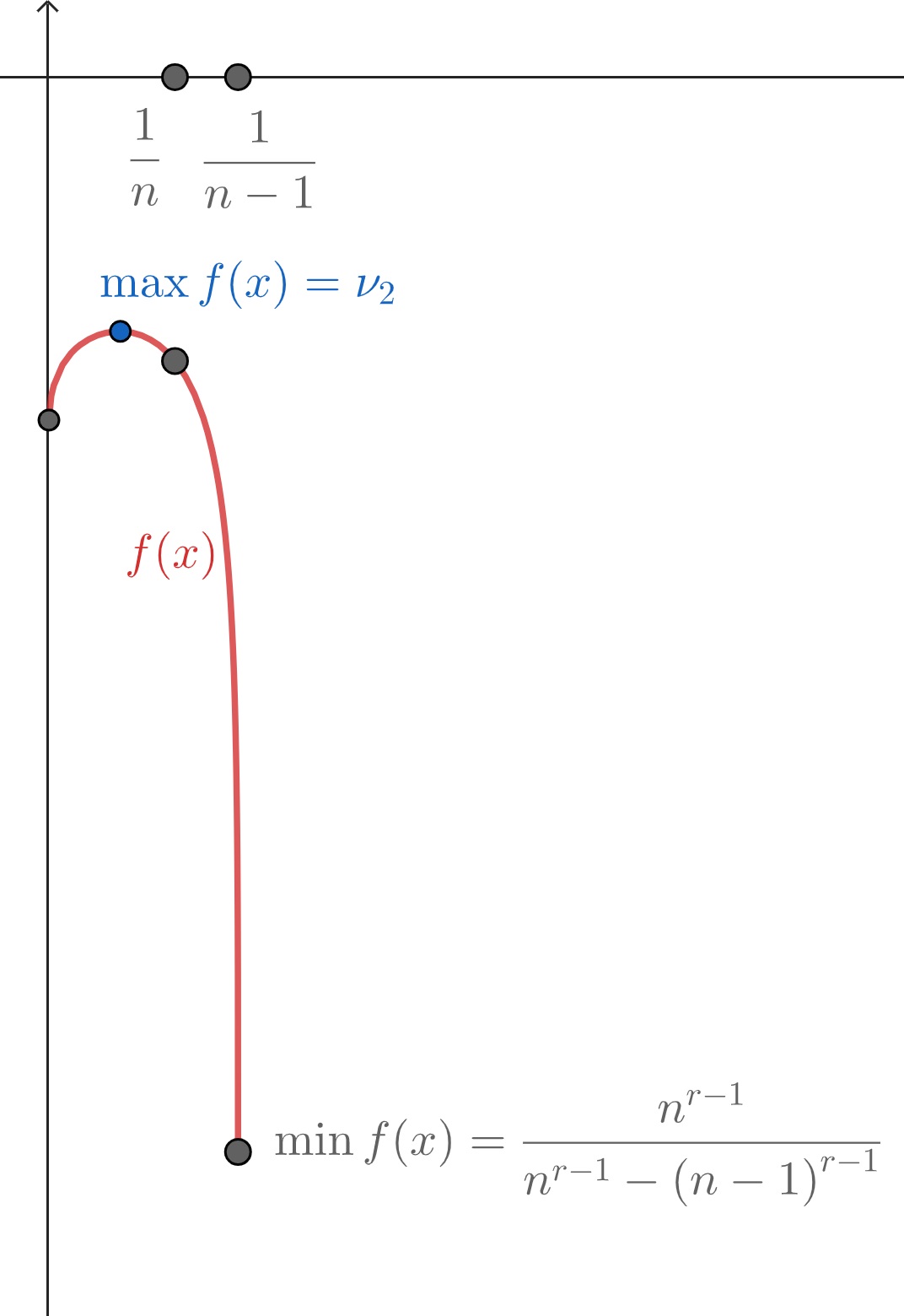}\ \includegraphics[scale=.4]{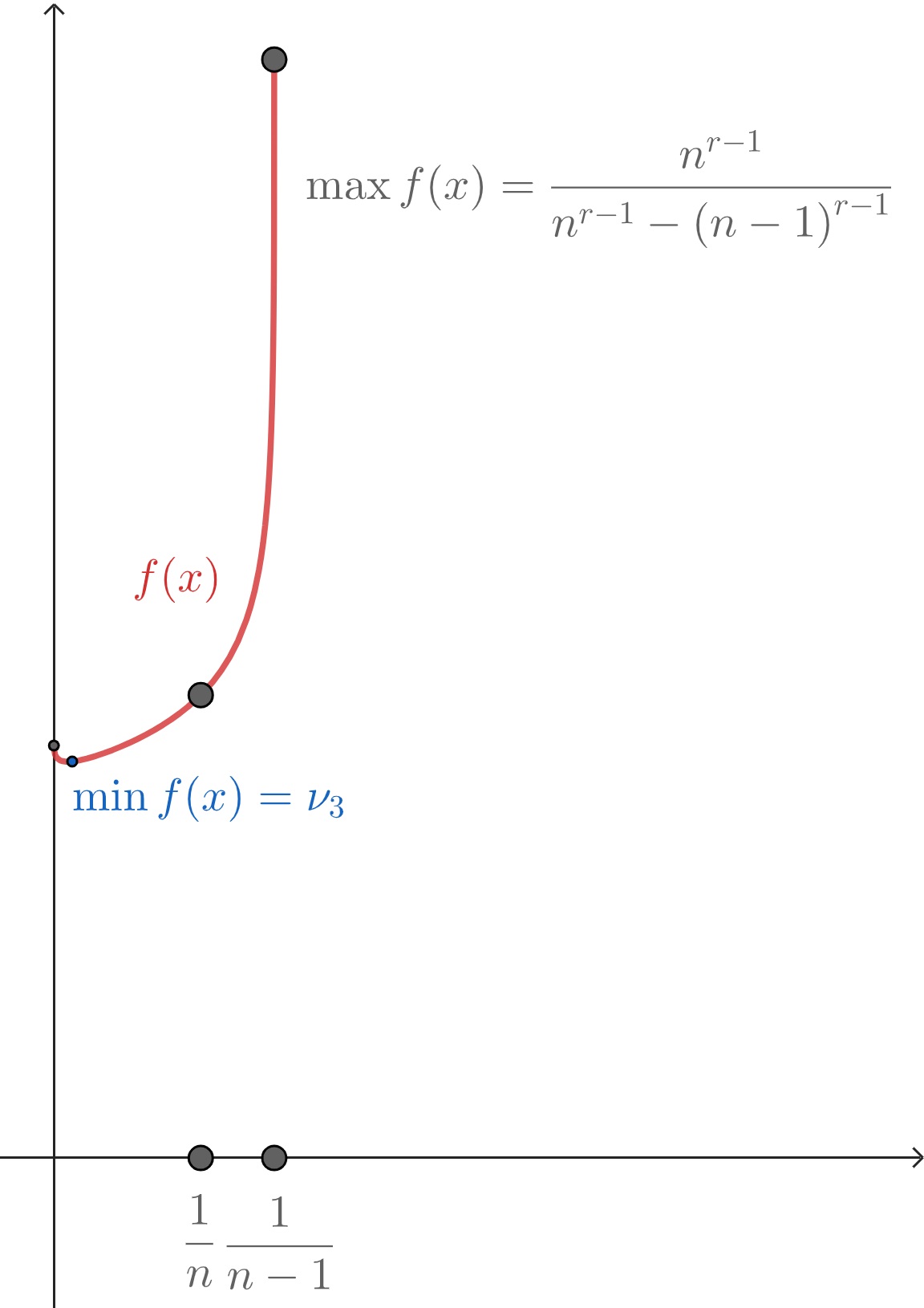}}
\centerline{\includegraphics[scale=.4]{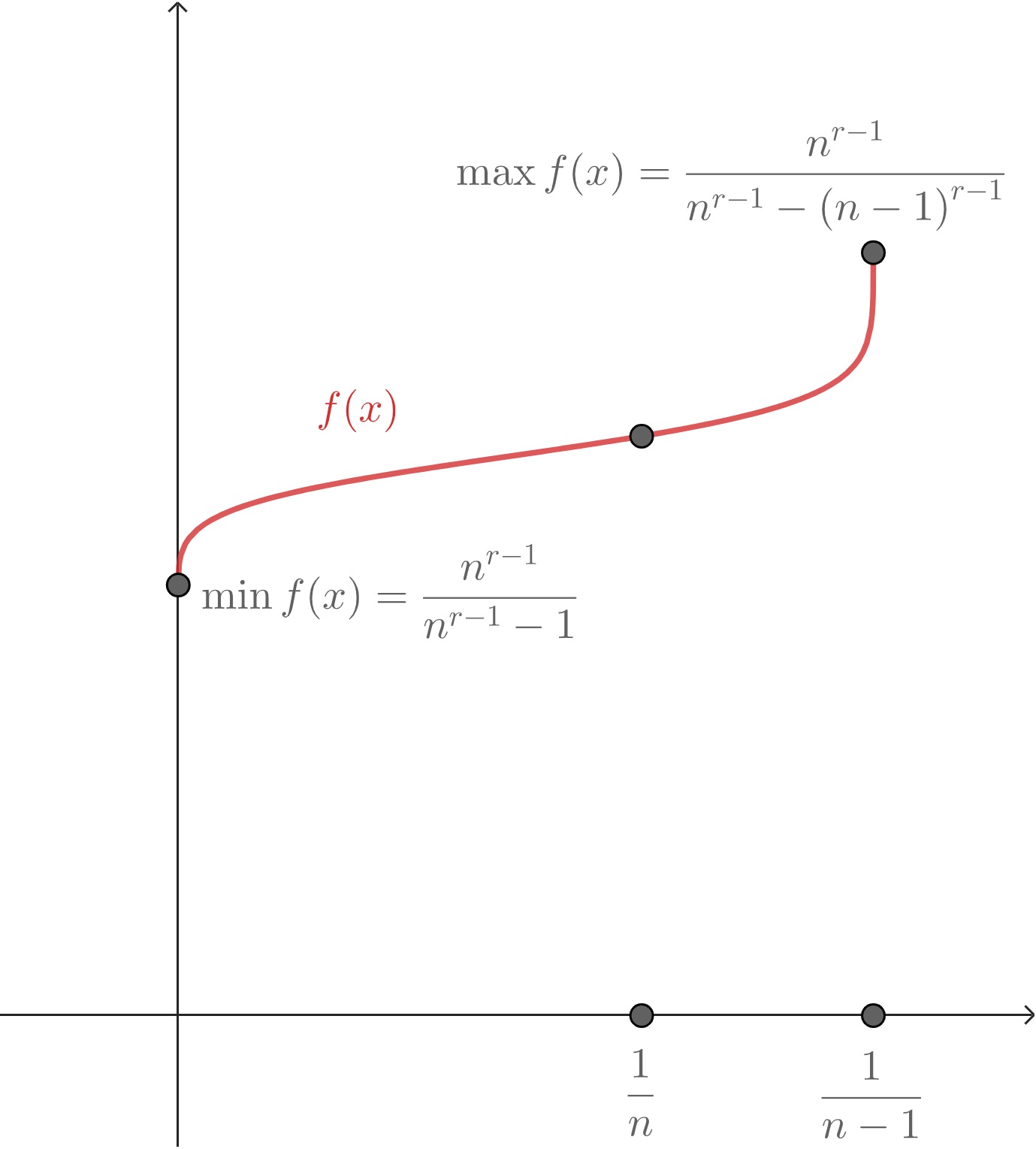}\ \includegraphics[scale=.4]{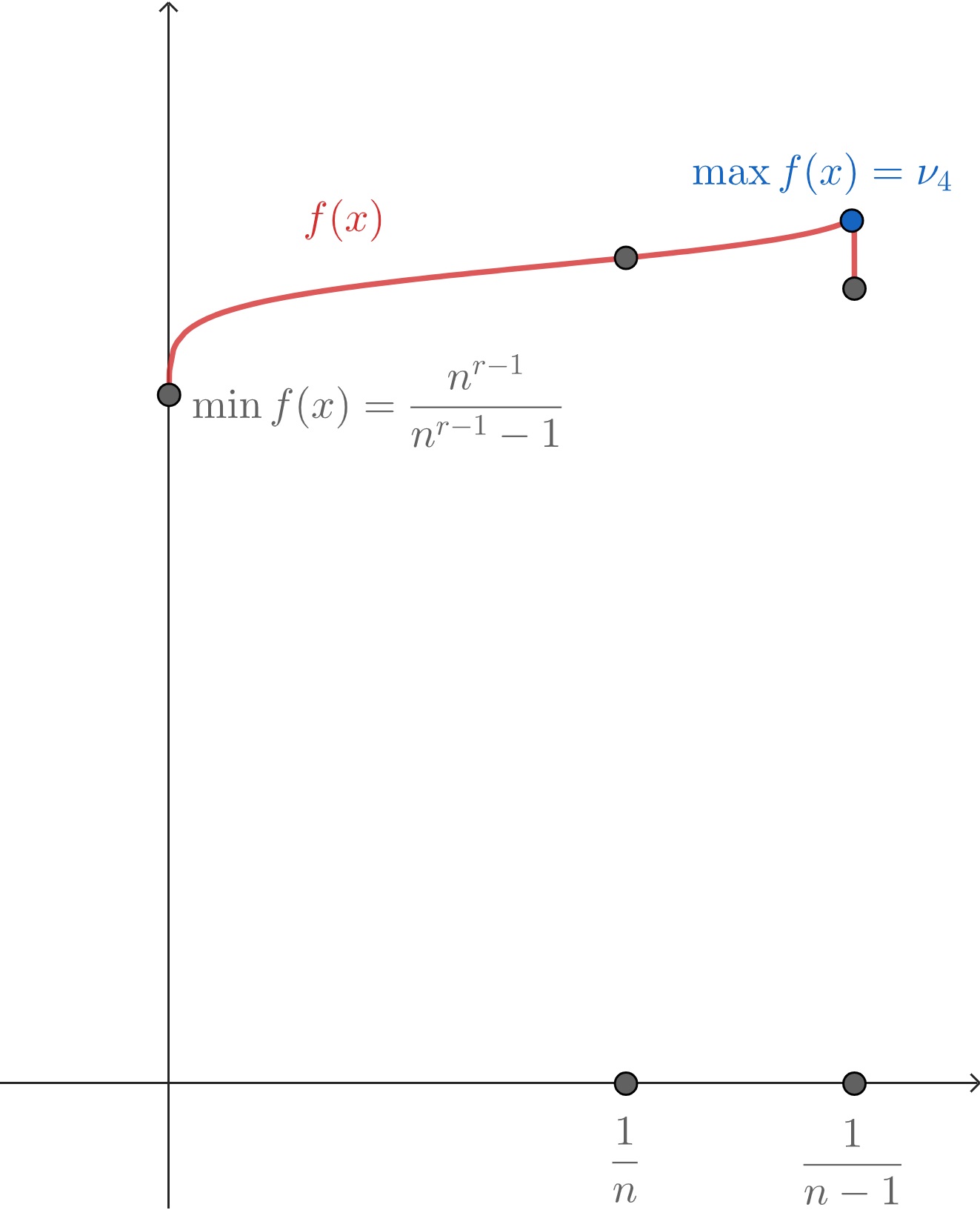}}
\label{fig4}
\caption{The graph of $f(x)$ (red), and its extreme values in interval $\left(0,\frac{1}{n-1}\right)$. Top: $r<0$;\\ 
Middle left: $0<r<1$; Middle right:  $1<r<2$ and $n<\frac{r}{r-1}$;\\Bottom left: $1<r\le2$ and $n\ge\frac{r}{r-1}$; and
$2\le r$ and $n\ge r$;\\ Bottom right: $2<r$ and $n<r$.}
\end{figure}

\begin{itemize}
\item If $r<0$, then function $f(x)=\frac{g(x)-\frac{1}{n}}{p(x)-\frac{1}{n}}$ is decreasing in interval $\left(0,\mu_1\right)$ from its maximum 1 at point $x=0$. In interval $\left(\mu_1, \frac{1}{n-1}\right)$ function $f(x)$ decreases to minimum $\nu_1$ and then increases reaching its maximum 1 at point $x=\frac{1}{n-1}$.

\item If $0<r<1$, then function $f(x)$ is decreasing in interval $\left(\mu_2, \frac{1}{n-1}\right)$ reaching its minimum $\frac{n^{r-1}}{n^{r-1}-(n-1)^{r-1}}$ at point $x=\frac{1}{n-1}$. In interval $\left(0,\mu_2\right)$ function $f(x)$ increases to maximum $\nu_2$ and then decreases.

\item If $1<r<2$ and $n<\frac{r}{r-1}$, then function $f(x)$ is increasing in interval $\left(\mu_3, \frac{1}{n-1}\right)$ reaching its maximum $\frac{n^{r-1}}{n^{r-1}-(n-1)^{r-1}}$ at point $x=\frac{1}{n-1}$. In interval $\left(0,\mu_3\right)$ function $f(x)$ decreases to minimum $\nu_3$ and then increases.

\item If $1<r\le2$ and $n\ge\frac{r}{r-1}$, or if $2\le r$ and $n\ge r$, then function $f(x)$ is increasing in $\left(0, \frac{1}{n-1}\right)$ from its minimum $\frac{n^{r-1}}{n^{r-1}-1}$ at point $x=0$ to its maximum $\frac{n^{r-1}}{n^{r-1}-(n-1)^{r-1}}$ at point $x=\frac{1}{n-1}$.

\item If $2\le r$ and $n<r$, then function $f(x)$ is increasing in interval $\left(0,\mu_4\right)$. In interval $\left(\mu_4, \frac{1}{n-1}\right)$ function $f(x)$ increases to maximum $\nu_4$ and then decreases.

\end{itemize}
The reason for the presence of these many cases can be explained informally with two pairs of close extrema points of $f(x)$, which can be in and out of interval $\left(0, \frac{1}{n-1}\right)$ depending on the choice of $n$ and $r$. In Figure 5, case $r=\frac{7}{5}$ and $n=3$ ($1<r<2$ and $n<\frac{r}{r-1}$) is shown.

\begin{figure}[htbp]
\centerline{\includegraphics[scale=.3]{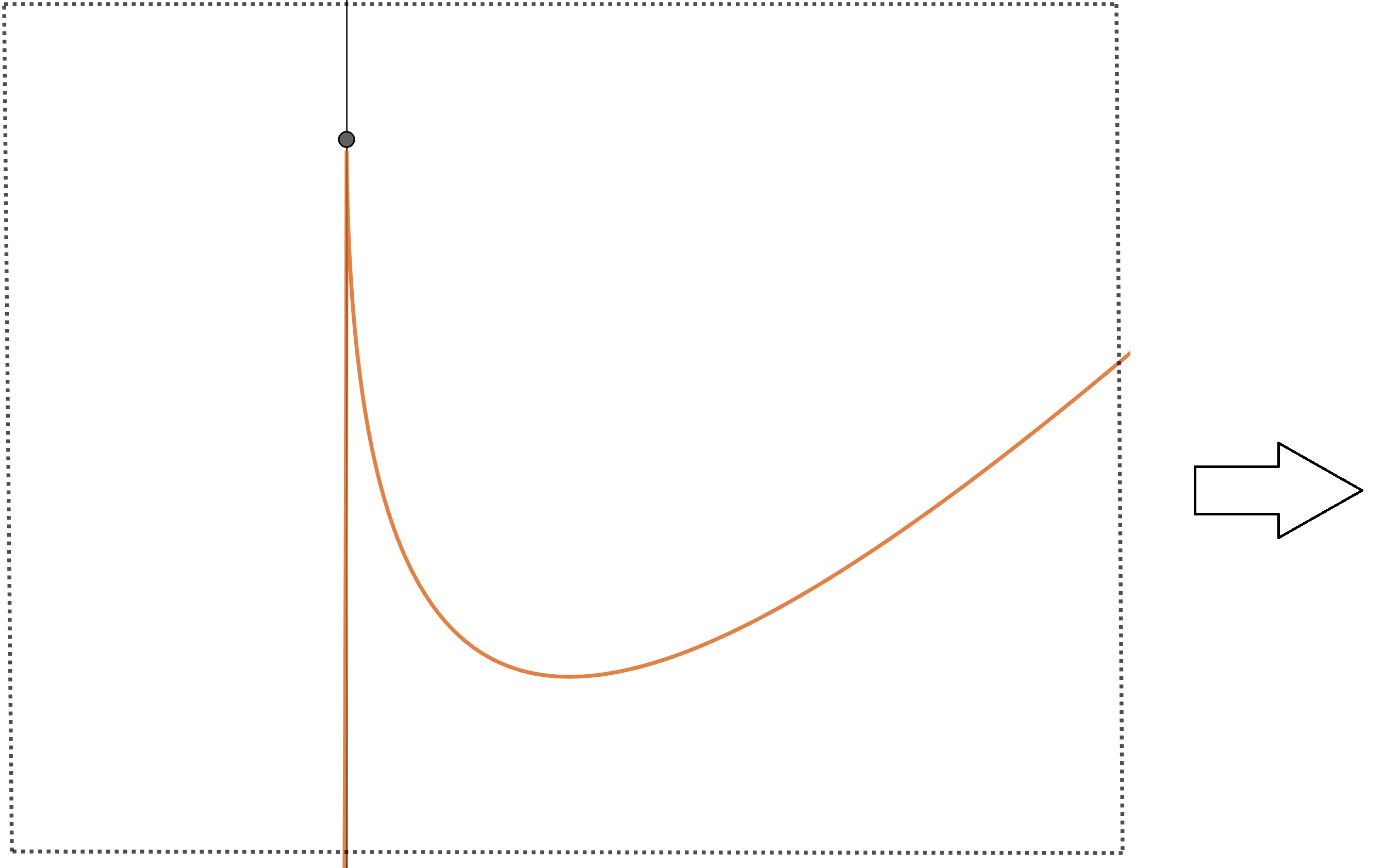}\includegraphics[scale=.4]{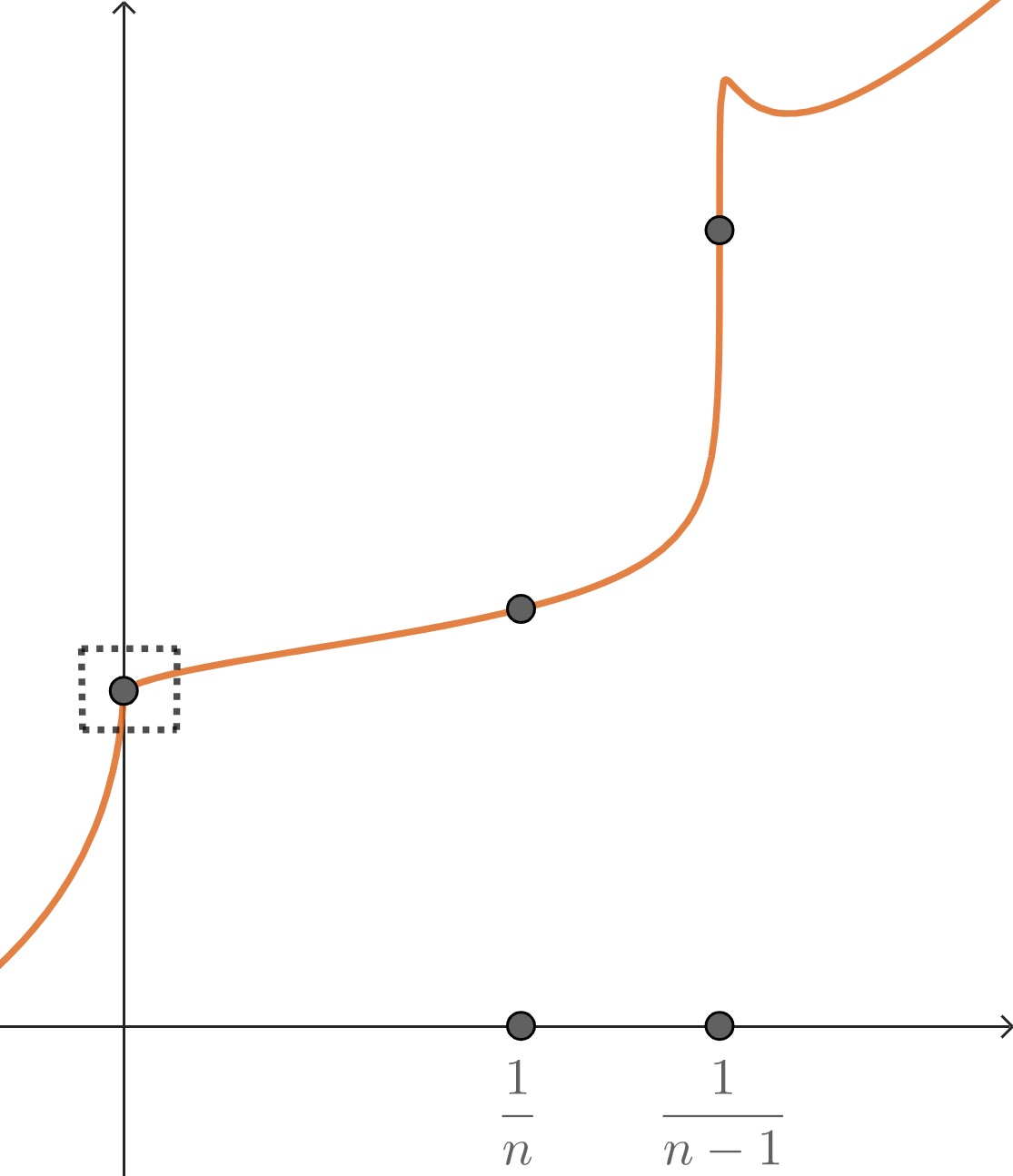}}
\label{fig5}
\caption{The graph of $f(x)$ can be extended beyond interval $0\le x\le\frac{1}{n-1}$ when, for example, $r=\frac{7}{5}$ and $n=3$. Note the two pairs of extrema points only one of which is in interval $\left(0, \frac{1}{n-1}\right)$ and can be seen when zoomed in significantly.}
\end{figure}

By using the fact that if $$x_{1}=x_{2}=\ldots =x_{n-1}=x,\ x_{n}=1-(n-1)x,$$ where $0\le x\le\frac{1}{n-1}$, then $\frac{A_n-G_n}{P_\alpha-G_n}=-\frac{1}{\frac{1}{f(x)}-1}$, and therefore using the results for $f(x)$ we can prove the bounds for $\frac{A_n-G_n}{P_\alpha-G_n}$ in the following theorem. It remains to consider the only remaining case.

\noindent \textbf{The case $x_1=0$.} If $\alpha<0$, then it is easy to check that $\lim_{x_1\rightarrow 0^+}\frac{A_n-G_n}{P_\alpha-G_n}=-\infty$. If $x_1=0$, then $G_n=0$, and therefore $\frac{A_n-G_n}{P_\alpha-G_n}=\frac{A_n}{P_\alpha}$. 
If $0<\alpha<1$ and $x_1=0$, then one can easily prove that $\left(\frac{n}{n-1}\right)^{\frac{1}{\alpha}-1}\le\frac{A_n}{P_\alpha}\le n^{\frac{1}{\alpha}-1}$. Similarly, if $\alpha>1$ and $x_1=0$, then one can prove that $n^{\frac{1}{\alpha}-1}\le\frac{A_n}{P_\alpha}\le \left(\frac{n}{n-1}\right)^{\frac{1}{\alpha}-1}$.

Thus we proved the following theorem which generalizes the results in \cite{wu0}, p. 646-647 and \cite{wen2}, p. 134 and its references.

\textbf{Theorem 1.} 
\begin{enumerate}

\item \textit{If $\alpha<0$, then $-\infty<\frac{A_n-G_n}{P_\alpha-G_n}\le\omega_1=\frac{\nu_1}{\nu_1-1}$, where $\nu_1$ is the minimum of $f(x)$ in interval $\left(\frac{1}{n}, \frac{1}{n-1}\right)$.}

\item  \textit{If $\alpha>1$, then $\omega_2=\frac{\nu_2}{\nu_2-1}\le\frac{A_n-G_n}{P_\alpha-G_n}\le\left(\frac{n}{n-1}\right)^{\frac{1}{\alpha}-1}$, where $\nu_2$ is the maximum of $f(x)$ in interval $\left(0,\frac{1}{n}\right)$.}

\item \textit{If $\frac{1}{2}<\alpha<1$ and $n<\frac{1}{1-\alpha}$, then $\left(\frac{n}{n-1}\right)^{\frac{1}{\alpha}-1}\le\frac{A_n-G_n}{P_\alpha-G_n}\le\omega_3=\frac{\nu_3}{\nu_3-1}$, where $\nu_3$ is the minimum of $f(x)$ in interval $\left(0,\frac{1}{n}\right)$.}

\item \textit{If $\frac{1}{2}\le\alpha<1$ and $n\ge\frac{1}{1-\alpha}$, or if  $0<\alpha\le\frac{1}{2}$ and $n\ge \frac{1}{\alpha}$, then $\left(\frac{n}{n-1}\right)^{\frac{1}{\alpha}-1}\le\frac{A_n-G_n}{P_\alpha-G_n}\le n^{\frac{1}{\alpha}-1}$.}

\item \textit{If $0<\alpha<\frac{1}{2}$ and $n<\frac{1}{\alpha}$, then $\omega_4=\frac{\nu_4}{\nu_4-1}\le\frac{A_n-G_n}{P_\alpha-G_n}\le n^{\frac{1}{\alpha}-1}$, where $\nu_4$ is the maximum of $f(x)$ in interval $\left(\frac{1}{n}, \frac{1}{n-1}\right)$.}

\end{enumerate}

The obtained results can also be interpreted as the best constant inequalities generalizing AM-GM inequality. For example, part 4 of Theorem above implies that if $\frac{1}{2}\le\alpha<1$ and $n\ge\frac{1}{1-\alpha}$, or if  $\alpha\le\frac{1}{2}$ and $n\ge \frac{1}{\alpha}$, then $\eta= n^{\frac{1}{\alpha}-1}$ and $\delta=\left(\frac{n}{n-1}\right)^{\frac{1}{\alpha}-1}$ are the best constants for inequalities
$$
\delta P_{\alpha}+(1-\delta)G_n\le A_n\le\eta P_{\alpha}+(1-\eta)G_n.
$$
These inequalities can be interpreted as sublinear and superlinear bounds for the arithmetic mean in terms of the power mean and the geometric mean. Note also that if all $x_i$ ($i=1,2,\ldots,n$) are replaced by $x_i^{\frac{1}{\alpha}}$, then the last double inequality can be written as
$$
\delta A_n^r +(1-\delta)G_n^r\le P_{r}^r\le\eta A_n^r+(1-\eta)G_n^r,
$$
where $r=\frac{1}{\alpha}$. Using the same substitutions $x_i\rightarrow x_i^{\frac{1}{\alpha}}$ ($i=1,2,\ldots,n$) we can see that the bounds in Theorem above are all true for $\frac{P_{r}^r-G_n^r}{A_n^r-G_n^r}$, too (cf. \cite{wu0}, p. 645, Corollary 4). If $r=n$, then we obtain the inequality in \cite{wu0} (Corollary 4, case $r=n$), for which an independent proof is given in \cite{aliyev2}. As a problem for further exploration, it would be interesting to find the extreme values of $\frac{P_\alpha-P_\beta}{P_\gamma-P_\delta}$ or more special $\frac{P_\alpha-P_\beta}{P_\gamma-P_\beta}$, for all possible values of the parameters. The answer to this question can possibly explain why $\alpha=\frac{1}{2}$ is special in the above theorem.

In \cite{wen2}, p. 135 (Theorem 2) it was proved that if $\alpha>1$, then $\omega_2\ge \frac{r}{n-1} \cdot\left(\frac{(2-r)(n-1)}{1+(1-r)(n-1)}\right)^{2-r}$, where $r=\frac{1}{\alpha}$. It seems that this bound works for $\alpha<0$ as well and possibly for some of the other cases, too. If $\omega_1\le \frac{r}{n-1} \cdot\left(\frac{(2-r)(n-1)}{1+(1-r)(n-1)}\right)^{2-r}$ is proved, then for $\alpha=-1$ it gives $\omega_1\le -{27} \cdot\frac{\left(n-1\right)^{2}}{\left(2n-1\right)^{3}}$, which can be substantial improvement for the bound $\omega_1\le -\frac{1}{n-1}$. The last inequality follows from the inequality $\frac{A_n-G_n}{A_n-H_n}>\frac{1}{n}$ proved by N. Lord in \cite{lord}. On the other hand $\omega_2,\omega_4\le \lim_{x\rightarrow\frac{1}{n}}\frac{\frac{1}{n}-g(x)}{p(x)-g(x)}=r$ and similarly $\omega_1,\omega_3\ge r$. It would be interesting to find sharper bounds for $\omega_i$ for each case of $\alpha$ (cf. \cite{aliyev1}, Theorem 1).

We will now study how the bounds in Theorem 1 change as $n$ increases. The monotonicity and convergence of $n^{\frac{1}{\alpha}-1}$ and $\left(\frac{n}{n-1}\right)^{\frac{1}{\alpha}-1}$ are straightforward so, we will focus on the monotonicity and the convergence of $\omega_1,\omega_2,\omega_3,$ and $\omega_4$.

\textbf{Theorem 2.} 
\begin{enumerate}

\item \textit{If $\alpha<0$, then $\omega_1$ increases as $n$ increases and $\lim_{n\rightarrow \infty}\omega_1=0$.}

\item  \textit{If $\alpha>1$, then $\omega_2$ decreases as $n$ increases and $\lim_{n\rightarrow \infty}\omega_2=0$.}

\item \textit{If $\frac{1}{2}<\alpha<1$, then $\omega_3$ increases as $n$ increases in interval $0<n<\frac{1}{1-\alpha}$.}

\item \textit{If $0<\alpha<\frac{1}{2}$, then $\omega_4$ decreases as $n$ increases in interval $0<n<\frac{1}{\alpha}$.}

\end{enumerate}

\textit{Proof.} Monotonicity: Let $h(x_1,\ldots,x_n)=\frac{A_n-G_n}{P_\alpha-G_n}$. It is sufficient to show that if $\alpha<1$ and $\alpha\ne 0$, then  $h(x_1,\ldots,x_n,G_n)\ge h(x_1,\ldots,x_n)$, and similarly, if $\alpha>1$, then $h(x_1,\ldots,x_n,G_n)\le h(x_1,\ldots,x_n)$ (cf. \cite{lord}). This is easily proved by noting
$$h(x_1,\ldots,x_n,G_n)=\frac{n(A_n-G_n)}{(n+1)\left(\left(\frac{nP_{\alpha}^{\alpha}+G_n^{\alpha}}{n+1}\right)^{\frac{1}{\alpha}}-G_n\right)}.$$
If $\alpha>1$, then the inequality $h(x_1,\ldots,x_n,G_n)\le h(x_1,\ldots,x_n)$ simplifies to
$$\left(\frac{nP_{\alpha}^{\alpha}+G_n^{\alpha}}{n+1}\right)^{\frac{1}{\alpha}}\ge\frac{nP_{\alpha}+G_n}{n+1},$$
which follows from the power mean inequality. The case $\alpha<1$ and $\alpha\ne 0$ is considered similarly.

Convergence: If $\alpha>1$, then $\omega_2\le n^{\frac{1}{\alpha}-1}$. Since in this case $\lim_{n\rightarrow \infty}n^{\frac{1}{\alpha}-1}=0$, $\lim_{n\rightarrow \infty}\omega_2=0$. If $-1\le\alpha<0$, then we can take $x_1=\cdots=x_{n-1}=1$ and $x_n=\frac{1}{n}$. Then $A_n=1-\frac{1}{n}+\frac{1}{n^2}$, $G_n=\frac{1}{\sqrt[n]{n}}$, and $P_{\alpha}=\left(\frac{n-1+\frac{1}{n^{\alpha}}}{n}\right)^{\frac{1}{\alpha}}$. In particular, $H_n=P_{-1}=\frac{n}{2n-1}$. Then $\lim_{n\rightarrow \infty}A_n=1$, $\lim_{n\rightarrow \infty}G_n=1$, $\lim_{n\rightarrow \infty}P_{\alpha}=1$ if $-1<\alpha<0$, and  $\lim_{n\rightarrow \infty}H_n=\frac{1}{2}$ ($\alpha=-1$). Therefore, if $-1\le\alpha<0$, then $\lim_{n\rightarrow \infty}\frac{A_n-G_n}{P_\alpha-G_n}=0$. Since $0\ge\omega_1\ge h\left(\underbrace{1,\ldots,1}_{n-1},\frac{1}{n}\right)$, we obtain $\lim_{n\rightarrow \infty}\omega_1=0$. If $\alpha<-1$, then $P_\alpha\le H_n$, and therefore 
$$0\ge\omega_1\ge\frac{A_n-G_n}{P_\alpha-G_n}\ge\frac{A_n-G_n}{H_n-G_n}.$$ It is already proved that for $n$-tuple $\left(\underbrace{1,\ldots,1}_{n-1},\frac{1}{n}\right)$ we have $\lim_{n\rightarrow \infty}\frac{A_n-G_n}{H_n-G_n}=0$. Consequently, in this case again $\lim_{n\rightarrow \infty}\omega_1=0$. Proof is complete.

\section{Conclusion} In the paper the best constants for the inequalities $C_1\le \frac{A_n-G_n}{P_\alpha-G_n}\le C_2$ are studied.  If these best constants exist, then their monotonicity and convergence properties are studied. The obtained results complete the known particular cases of the choice of $\alpha$, which were studied in earlier works.

\section*{Acknowledgments}

\section{Declarations}
\textbf{Ethical Approval.}
Not applicable.
 \newline \textbf{Competing interests.}
None.
  \newline \textbf{Authors' contributions.} 
Not applicable.
  \newline \textbf{Funding.}
This work was completed with the support of ADA University Faculty Research and Development Fund.
  \newline \textbf{Availability of data and materials.}
Not applicable

\end{document}